%% file: main.tex
\documentclass[11pt]{article}

\pdfoutput=1
\input{preamble}

\title{\texorpdfstring{$\ell^\infty$-cohomology}{L-infinity cohomology}: amenability, relative hyperbolicity, isoperimetric inequalities and undecidability}
\author{Francesco Milizia\\[-1pt]
	\footnotesize{\textit{Scuola Normale Superiore, Palazzo della Carovana,}}\\[-3pt]
	\footnotesize{\textit{Piazza dei Cavalieri, 7, 56126 Pisa, IT}}\\[-1pt]
	\footnotesize{\texttt{francesco.milizia@sns.it}}}
\date{}

\hypersetup{
    pdfsubject={AMS Subject Classification: 20J06, 55N25, 20F10.},
    pdfkeywords={group cohomology; amenable groups; relatively hyperbolic groups; algorithmic undecidability.}
}

\begin{document}
    \maketitle
    
    \vspace{-20pt}
    \begin{abstract}
        We revisit Gersten's $\ell^\infty$-cohomology of groups and spaces, removing the finiteness assumptions required by the original definition while retaining its geometric nature.
        Mirroring the corresponding results in bounded cohomology, we provide a characterization of amenable groups using $\ell^\infty$-cohomology, and generalize Mineyev's characterization of hyperbolic groups via $\ell^\infty$-cohomology to the relative setting.
        We then describe how $\ell^\infty$-cohomology is related to isoperimetric inequalities.
        We also consider some algorithmic problems concerning $\ell^\infty$-cohomology and show that they are undecidable.
        In an appendix, we prove a version of the de Rham's theorem in the context of $\ell^\infty$-cohomology.
    \end{abstract}

    \section{Introduction}
    
    Gersten introduced \Lcoh{} as a tool to obtain lower bounds for the Dehn function of finitely presented groups \cite{Gersten1992}.
    His approach has a geometric-topological flavor, and starts with the definition of \Lcoh{} for \CWs{}.
    If $X$ is a connected \CW{} with a finite $m$-skeleton, then Gersten defines the \Lcoh{} of $X$, only in degrees $\le m$, as the cohomology of the complex of bounded cellular cochains on the universal cover of $X$.
    Here, the coefficients are given by a normed real vector space $V$, usually the real numbers with the absolute-value norm, and a cochain is bounded if there is a bound on the norm of the values it assigns to the cells.
    The finiteness of the $m$-skeleton of $X$ guarantees that the usual coboundary map sends bounded cochains of degree $<m$ to bounded cochains.
    Then, Gersten defines the \Lcoh{} of a group $G$ as the \Lcoh{} of any cellular $K(G,1)$; it is defined in degree $m$ if there is a cellular $K(G,1)$ with a finite $m$-skeleton, \ie, if $G$ is of type $F_m$.
    For example, it is defined in degree $2$ if $G$ is finitely presentable.
    
    Our terminology differs slightly from the one used by Gersten and other authors: the \Lcoh{} of a \CW{} $X$, defined as above, should be called, according to Gersten, the \Lcoh{} (or \emph{bounded valued cohomology}) \emph{of the universal cover} of $X$.
    We denote it by $H_\Linfs^\bullet(X;V)$.
    For the \Lcoh{} of a group $G$, we use the same notation with $G$ in place of $X$.
    
    Wienhard \cite[Proposition 5.2]{Wienhard2012} and Blank \cite[Theorem 6.3.5]{Blank2015} observed that $H_\Linfs^\bullet(G;V)$ is canonically isomorphic to the (ordinary) cohomology of $G$ with coefficients in the $\R[G]$-module
    \[ \Linf(G,V) = \{ f:G \to V \text{ with bounded image}\}. \]
    The latter is always defined in every degree.
    We take it as our definition of \Lcoh{}, thus removing the finiteness condition appearing in Gersten's construction.
    
    More generally, if $X$ is a connected space with fundamental group $G$, we define the \Lcoh{} of $X$ as the cohomology of $X$ with coefficients in $\Linf(G,V)$.
    This definition has a simple description in terms of $V$-valued cochains, in a spirit that is closer to Gersten's construction.
    Let $X$ be a connected \CW{} with fundamental group $G$ and universal cover $\Ucover{X}$, and let $V$ be a normed vector space.
    Instead of considering uniformly bounded $V$-valued cochains on $\Ucover{X}$, we consider cochains bounded on every $G$-orbit of cells, where $G$ acts on $\Ucover{X}$ via covering transformations.
    When $X$ has a finite $m$-skeleton, for every $k \le m$ there is a finite number of orbits of $k$-dimensional cells; therefore, a $k$-cochain is uniformly bounded if and only if it is bounded on orbits, and Gersten's definition is recovered.
    Quite surprisingly, the description via cochains bounded on orbits, which is more carefully explained in \Cref{sec:definition} and is the leitmotif of this exposition, also works with singular cochains in place of cellular ones; this is not true if one imposes uniform boundedness, as in the original definition by Gersten.
    
    There is a natural \emph{comparison map} $\iota^\bullet:H^\bullet(G;V) \to H^\bullet_\Linfs(G;V)$ from the ordinary to the \Lcoh{} of a group, whose definition is recalled in \Cref{sec:comparison}.
    When $G$ is amenable and $V$ is a dual normed $\R[G]$-module, $\iota^k$ is injective for every $k \in \N$.
    This is one of the first facts that Gersten proved, at least for trivial real coefficients and for degrees $k$ such that $G$ is of type $F_k$, as soon as he introduced \Lcoh{} \cite[Theorem 10.13]{Gersten1992}.
    Wienhard then removed the finiteness condition \cite[Proposition 5.3]{Wienhard2012} as soon as she introduced the more general definition of \Lcoh{} of groups (see also \cite[Remark 5.3]{Frigerio2020} for a different proof).
    In \Cref{sec:amenability} we prove a converse of this result, obtaining the following characterization of amenability:
    
    \begin{theorem}
        \label{thm:intro_linf_characterizes_amenable}
        Let $G$ be a group.
        Then $G$ is amenable if and only if the comparison map $\iota^1:H^1(G;V) \to H^1_\Linfs(G;V)$ is injective for every dual normed $\R[G]$-module $V$.
        Moreover, if $G$ is amenable, then $\iota^k:H^k(G;V) \to H^k_\Linfs(G;V)$ is injective for every $k \in \N$ and every $V$ as above.
    \end{theorem}
    
    Another prominent class of groups in geometric group theory is that of hyperbolic groups.
    A series of papers by Gersten \cite{Gersten1996a, Gersten1996b, Gersten1998} and Mineyev \cite{Mineyev1999, Mineyev2000}
    led to the following characterization of hyperbolicity:
    
    \begin{theorem}[\cite{Mineyev2000}]
        \label{thm:intro_linf_characterizes_hyperbolic}
        Let $G$ be a finitely presented group.
        Then $G$ is hyperbolic if and only if $H_\Linfs^2(G;V) = 0$ for every normed vector space $V$.
        Moreover, if $G$ is hyperbolic, then $H_\Linfs^k(G;V) = 0$ for every $k \ge 2$ and every normed vector space $V$.
    \end{theorem}
    
    In \Cref{sec:relative_hyperbolicity} we consider group pairs $(G,\mathcal{H})$, where $G$ is a group and $\mathcal{H}$ is a collection of subgroups of $G$.
    We define $H_\Linfs^\bullet(G,\mathcal{H};V)$ using the classical relative cohomology theory of Bieri and Eckman \cite{BE1978} with coefficients in $\Linf(G,V)$.
    When the pair $(G,\mathcal{H})$ satisfies a certain finiteness condition, this coincides with the definition given by Gautero and Heusener \cite{GH2009}.
    We obtain the following generalization of \Cref{thm:intro_linf_characterizes_hyperbolic}:
    
    \begin{theorem}
        \label{thm:intro_relative_hyperbolicity}
        Let $(G,\mathcal{H})$ be a group pair admitting a finite relative presentation, with $G$ finitely generated.
        Then the following conditions are equivalent:
        \begin{enumerate}[label=(\roman*)]
            \item \label{it:intro_rel_hyp} The pair $(G,\mathcal{H})$ is relatively hyperbolic;
            \item \label{it:intro_linf_null} $H_\Linfs^2(G,\mathcal{H};V) = 0$ for every normed vector space $V$;
            \item \label{it:intro_strong_null} $H_\Linfs^2(G,\mathcal{H};\R)$ vanishes strongly.
        \end{enumerate}
        Moreover, if $(G,\mathcal{H})$ is relatively hyperbolic, then $H_\Linfs^k(G,\mathcal{H};V) = 0$ for every normed vector space $V$ and every $k \ge 2$, and $H_\Linfs^k(G,\mathcal{H};\R)$ vanishes strongly for every $k \ge 2$.
    \end{theorem}
    
    The notion of strong vanishing (which is recalled in \Cref{sec:relative_hyperbolicity}) was already considered in the non-relative setting by Gersten \cite{Gersten1996b}, who proved that $H_\Linfs^2(G;\R)$ vanishes strongly if and only if $G$ is hyperbolic.
    Gautero and Heusener considered strong vanishing in the relative setting and proved that if $(G,\mathcal{H})$ is relatively hyperbolic then $H_\Linfs^2(G,\mathcal{H};\R)$ vanishes strongly, but proved the converse implication only in some special cases.
    
    Another cohomology theory that characterizes amenable and hyperbolic groups is bounded cohomology.
    We recall its definition in \Cref{sec:comparison}.
    A connection between bounded and \Lcoh{} is established by the following proposition, first proved by Gersten for groups of type $F_k$ and trivial real coefficients, and then by Wienhard and Blank in general.
    
    \begin{proposition}[\cite{Gersten1992, Wienhard2012, Blank2015}]
        \label{prop:null-composition_intro}
        Let $G$ be a group, $V$ a normed $\R[G]$-module and $k$ a positive integer.
        Then the composition of the comparison maps
        \begin{equation}
            \label{eq:sequence}
            \begin{tikzcd}
                H_b^k(G;V) \ar[r,"c^k"] & H^k(G;V) \ar[r,"\iota^k"] & H_\Linfs^k(G;V)
            \end{tikzcd}
        \end{equation}
        is the zero map.
    \end{proposition}
    
    In \Cref{sec:comparison} we extend this result to the cohomology of topological spaces (\Cref{cor:null-composition}).
    Our proof relies on an application of the fundamental theorem of Gromov about the vanishing of the bounded cohomology of simply connected spaces.
    A version of this result also holds for group pairs, in degrees $k \ge 2$ (\Cref{cor:null-composition_relative}).
    
    An interesting problem, posed in \cite[Question 8]{Wienhard2012} and \cite[Question 6.3.10]{Blank2015}, asks to investigate under what circumstances the sequence \eqref{eq:sequence} is exact at the central term.
    To tackle this type of questions it is useful to have some criterion to establish that a certain cochain represents the trivial class in \Lcoh{}.
    The following result, proved in \Cref{sec:isoperimetric}, characterizes the (cellular) representatives of the trivial class in \Lcoh{} as those cocycles satisfying a linear isoperimetric inequality for cellular boundaries.
    
    \begin{theorem}
        Let $X$ be a connected \CW{} with universal cover $\Ucover{X}$.
        Let $\alpha \in C^k(\Ucover{X},\R)$ be a cellular cocycle.
        Suppose that there is a real number $\Lambda \ge 0$ such that
        $\abs{\alpha(c)}\ \le\ \Lambda \cdot \norm{\Bd c}_1$
        for every cellular chain $c \in C_k(\Ucover{X}, \R)$.
        Then $\alpha$ is bounded on orbits and $[\alpha] = 0 \in H_\Linfs^k(X;\R)$.
        Moreover, if the $(k-1)$-skeleton of $X$ is finite, then the converse implication also holds.
    \end{theorem}
    
    This characterization is implicitly contained in the works of Gersten \cite{Gersten1996b} and Sikorav \cite{Sikorav2001}; we present it explicitly here to make it more easily accessible, especially in view of an application by Frigerio and Sisto \cite{Frigerio2020}: the construction of a finitely generated group $G$ such that the sequence $H_b^2(G;\R) \to H^2(G;\R) \to H_\Linfs^2(G;\R)$ is not exact\footnote{After earlier versions of this paper appeared on arXiv, a finitely presented group with the same property has been found \cite{AM2022}, settling a conjecture of Gromov \cite[page 138]{Gromov1993}.}.
    In \Cref{sec:isoperimetric} we briefly describe this and other applications.
    
    In \Cref{sec:uncomputability} we consider the problem of computing \Lcoh{} algorithmically.
    It is easy to devise an algorithm that computes the ordinary homology or cohomology of finite simplicial complexes, with coefficients in $\Z$ or $\R$.
    By contrast, there is no algorithm to decide, given a finite presentation $\langle S|R \rangle$ of a group, whether $H_2(\langle S|R \rangle;\Z)$ vanishes or not \cite[Theorem 4]{Gor1995}.
    The argument in \cite{Gor1995} also shows that the vanishing of $H^2(\langle S|R \rangle;\R)$ is undecidable.
    We prove that the analogous problem for \Lcoh{} is undecidable, both for finitely presented groups and finite simplicial complexes, in every positive degree:
    \begin{theorem}
        \label{thm:intro_undecidable}
        For every $n \ge 1$ the following algorithmic problems are undecidable:
        \begin{itemize}
            \item Given a finite presentation $\langle S|R\rangle$, decide whether $H_\Linfs^n(\langle S|R\rangle;\R) = 0$;
            \item Given a finite simplicial complex $X$, decide whether $H_\Linfs^n(X;\R) = 0$.
        \end{itemize}
    \end{theorem}
    The same algorithmic problems have recently been considered in the context of bounded cohomology: for every $n \ge 2$, the vanishing of bounded cohomology in degree $n$ is algorithmically undecidable, both for finitely presented groups and finite simplicial complexes \cite{FLM2021}.
    
    In \Cref{sec:differential} we consider the \Lcoh{} of smooth manifolds and present a version of the de Rham's theorem in the context of \Lcoh{}.
    In particular, by taking differential forms on the universal cover which are bounded on orbits in a suitable sense, we define the \Ldrc{} of a manifold, and then prove that it is canonically isomorphic to the singular \Lcoh{}.
    After establishing a Poincaré lemma and a Mayer-Vietoris sequence for \Ldrc{}, the result follows from a standard local-to-global argument.
    Similar results have been obtained in \cite{Gol1988}, \cite{AttieBlock1998} and \cite{Mineyev1999}, but they involve differential forms bounded \emph{globally}; for this reason, they require a ``bounded geometry'' condition on the manifold, which is not needed in our version of the theorem.
    
    The results contained in \Cref{sec:differential} are not used in other parts of the paper, and vice versa.

    \subsection*{Acknowledgements}
    I would like to thank my PhD advisor Roberto Frigerio for many helpful comments and suggestions.

    \section{\texorpdfstring{\bm{$\ell^\infty$}-cohomology}{L-infinity cohomology} of groups and spaces}
    \label{sec:definition}
    
    Let $G$ be a group, and let $V$ be a normed $\R[G]$-module, \ie, a normed vector space on which $G$ acts on the left by linear isometries.
    We denote by
    \[\Linf(G,V) = \{ f:G \to V \ \mid \ f(G) \text{ is a bounded subset of } V\} \]
    the normed $\R[G]$-module of bounded functions from $G$ to $V$.
    The norm is given by $\norm{f} = \sup\{ \norm{f(g)} \ \mid \ g \in G\}$, and the group $G$ acts on $\Linf(G,V)$ in the following way: given a function $f \in \Linf(G,V)$ and an element $g \in G$, the function $g \cdot f \in \Linf(G,V)$ is defined by the formula $(g \cdot f)(h) = g \cdot (f(g^{-1}h))$, where $h$ varies in $G$.
    
    \begin{definition}
        \label{def:linfty-group}
        Let $G$ be a group and $V$ a normed $\R[G]$-module.
        We define the \Lcoh{} of $G$ with coefficients in $V$ as
        \begin{align*}
            H_\Linfs^\bullet(G; V) = H^\bullet(G; \Linf(G, V)),
        \end{align*}
        \ie, the ordinary cohomology of $G$ with coefficients in the $\R[G]$-module $\Linf(G, V)$.
        Notice that the norm on $\Linf(G,V)$ does not play any role in the definition.
    \end{definition}
    
    We introduce an analogous definition in the topological setting.
    We shall implicitly assume that any topological space under consideration is locally path-connected and semilocally simply connected.
    This conditions ensure the existence and uniqueness (up to isomorphism) of universal covers of connected spaces.
    Moreover, we restrict our definition to connected spaces.
    
    \begin{definition}
        \label{def:linfty-topological}
        Let $X$ be a connected topological space with a basepoint $x_0$.
        Denote by $G = \pi_1(X, x_0)$ the fundamental group.
        The \Lcoh{} of $X$ with coefficients in a normed $\R[G]$-module $V$ is defined as
        \begin{align*}
            H_\Linfs^\bullet(X; V) = H^\bullet(X; \Linf(G,V)).
        \end{align*}
        The latter denotes the singular equivariant cohomology \cite[Section VI.3]{Whitehead1978} of the universal cover of $X$ with coefficients in the $\R[G]$-module $\Linf(G,V)$.
    \end{definition}
    
    \begin{remark}
        Let $X$ be a $K(G,1)$, \ie, a connected topological space with a fixed isomorphism $\pi_1(X,x_0) \cong G$ and whose universal cover is contractible.
        Then, by general well-known theorems in cohomology, there is a canonical isomorphism $H^\bullet(G; \Linf(G,V)) \cong H^\bullet(X; \Linf(G,V))$.
    \end{remark}
    
    We now unravel the definition in order to get a more concrete understanding of the \Lcoh{} of a topological space $X$.
    Denote by $\Ucover{X}$ the universal cover of $X$.
    By definition, $H^\bullet(X; \Linf(G, V))$ is the cohomology of the following complex of vector spaces: 
    \[\begin{tikzcd}[column sep=small]
        0 \ar[r] & C^0(\Ucover{X};\Linf(G,V))^G \ar[r,"\Cbd"] & C^1(\Ucover{X};\Linf(G,V))^G \ar[r,"\Cbd"] & \cdots
    \end{tikzcd}.\]
    Here, $C^k(\Ucover{X};\Linf(G,V))$ is the $\R[G]$-module of singular $k$-cochains in $\Ucover{X}$ with coefficients in $\Linf(G,V)$, the superscript $^G$ denotes taking the $G$-invariant subspace, and $\Cbd$ is the usual coboundary operator.
    
    Let us recall how $G$ acts on the space of singular $k$-cochains.
    Denote by $\Simplex{k} \subseteq \R^{k+1}$ the standard $k$-simplex, spanned by the canonical basis of $\R^{k+1}$.
    The group $G$ acts on $\Ucover{X}$ by covering transformations, and this induces an action on the set of singular $k$-simplices $\Simplex{k}(\Ucover{X}) = \{\sigma:\Simplex{k} \to \Ucover{X} \text{ continuous}\}$.
    A singular $k$-cochain $\alpha$ is an arbitrary function $\alpha:\Simplex{k}(\Ucover{X}) \to \Linf(G,V)$.
    An element $g \in G$ transforms this function in the usual way:
    $(g \cdot \alpha)(\sigma) = g \cdot (\alpha(g^{-1}\cdot \sigma))$
    for every $\sigma \in \Simplex{n}(\Ucover{X})$.
    In particular, $\alpha$ belongs to the $G$-invariant subspace when the equality $\alpha(g \cdot \sigma) = g \cdot \alpha(\sigma)$ holds for every $g \in G$ and $\sigma \in \Simplex{n}(\Ucover{X})$, \ie, when $\alpha$ is equivariant.
    
    Any equivariant $k$-cochain $\alpha$ contains a lot of redundant information.
    One way to get rid of the redundancy is to replace every $\alpha(\sigma) \in \Linf(G,V)$ with its evaluation at the identity $1_G \in G$.
    Formally, define $\hat{\alpha} \in C^k(\Ucover{X};V)$ by setting $\hat{\alpha}(\sigma) = \alpha(\sigma)(1_G)$ for every $\sigma \in \Simplex{k}(\Ucover{X})$.
    Then $\alpha$ can be recovered from $\hat{\alpha}$, because the condition $g \cdot \alpha = \alpha$ gives
    \begin{align*}
        (\alpha(\sigma))(g)\ &=\ ((g \cdot \alpha)(\sigma))(g)\ =\ (g \cdot \alpha(g^{-1}\cdot\sigma))(g)\ =\ g \cdot (\alpha(g^{-1}\cdot\sigma)(1_G)) \\
        &=\ g \cdot \hat{\alpha}(g^{-1}\cdot\sigma).
    \end{align*}
    Not every cochain $\beta \in C^k(\Ucover{X};V)$ is of the form $\beta = \hat{\alpha}$ for some $\alpha$ as above.
    In fact, if $\beta = \hat{\alpha}$, then for any simplex $\sigma \in \Simplex{k}(\Ucover{X})$ the subset
    \[ \{ \hat\alpha(g \cdot \sigma)\ \mid\ g \in G\}\ =\ \{ g\cdot(\alpha(\sigma)(g^{-1}))\ \mid\ g \in G\}\ \subseteq\ V \]
    is bounded, because $\alpha(\sigma):G \to V$ is a bounded function and the action of $G$ on $V$ preserves the norm.
    
    \begin{definition}
        \label{defn:bounded_on_orbits}
        We say that a cochain $\beta \in C^k(\Ucover{X};V)$ is bounded on orbits if for every $\sigma \in \Simplex{k}(\Ucover{X})$ the subset $\{ \beta(g \cdot \sigma)\ \mid\ g \in G\} \subseteq V$ is bounded, and we denote by $C_\Linfs^k(\Ucover{X};V)$ the subspace containing such cochains.
    \end{definition}
    
    \begin{remark}
        The coboundary of a cochain bounded on orbits is bounded on orbits.
        In other words $C_\Linfs^\bullet(\Ucover{X};V)$ is a subcomplex of $C^\bullet(\Ucover{X};V)$.
    \end{remark}

    If $\beta \in C^k(\Ucover{X};V)$ is of the form $\beta = \hat{\alpha}$ for some $\alpha \in C^k(\Ucover{X};\Linf(G,V))^G$, then it is bounded on orbits.
    On the other hand, if $\beta$ is bounded on orbits, then $\beta = \hat{\alpha}$, where $\alpha$ is defined by the formula $(\alpha(\sigma))(g) = g \cdot \beta(g^{-1}\cdot\sigma)$ and is easily seen to be invariant (and with values in $\Linf(G,V)$).
    Notice also that the evaluation at $1_G$ is $\R$-linear and commutes with computing the coboundary, \ie, $\Cbd(\hat{\alpha}) = \widehat{(\Cbd\alpha)}$.
    This discussion proves the following:
    
    \begin{proposition}
        The \Lcoh{} $H_\Linfs^\bullet(X;V)$ is canonically isomorphic to the cohomology of the complex of vector spaces $C_\Linfs^\bullet(\Ucover{X};V)$.
    \end{proposition}
    
    \begin{remark}
        The complex $C_\Linfs^\bullet(\Ucover{X};V)$, and hence its cohomology, is clearly independent of the action of $G$ on the normed vector space $V$.
        This comes as no surprise, because different actions give rise to isomorphic $\R[G]$-module structures on $\Linf(G, V)$.
        In fact, let $V$ and $W$ be two normed $\R[G]$-modules and suppose that $\varphi:V \to W$ is a norm-preserving isomorphism (not necessarily a $G$-equivariant one).
        Then the map sending any $f \in \Linf(G,V)$ to the function $f' \in \Linf(G,W)$ defined as 
        $f'(h) = h \cdot \varphi(h^{-1} \cdot f(h))$ for every $h \in G$
        gives an isomorphism of $\R[G]$-modules.

        However, in \Cref{sec:comparison} we consider a natural ``comparison'' map from ordinary to \Lcoh{} which, in general, depends on the action of $G$ on the coefficients, and allows us to obtain in \Cref{sec:amenability} a characterization of amenable groups.
    \end{remark}

    If $X$ is a connected \CW{}, the entire discussion can be repeated using cellular cochains, replacing $\Simplex{k}(\Ucover{X})$ with the set of $k$-cells of $\Ucover{X}$.
    In particular, \Cref{defn:bounded_on_orbits} has an obvious adaptation for cellular cochains, and the usual isomorphism between singular and cellular cohomology implies that $H_\Linfs^\bullet(X; V)$ is canonically isomorphic to the cohomology of the complex of cellular cochains bounded on orbits.
    If $X$ has a finite $m$-skeleton, in degrees $\le m$ a cochain in $\Ucover{X}$ is bounded on orbits if and only if it is uniformly bounded.
    Therefore, Gersten's definition of \Lcoh{} is recovered.
    
    If $G$ is a group, the role of $\Simplex{k}(\Ucover{X})$ can be taken by the set $G^{k+1}$, on which $G$ acts diagonally.
    In particular, the vector spaces
    $C_\Linfs^k(G;V) = \{\alpha:G^{k+1} \to V \text{ bounded on orbits}\}$
    define a complex whose cohomology is canonically isomorphic to the \Lcoh{} of $G$ with coefficients in $V$.
    This is equivalent to the definition of \Lcoh{} of (finitely generated) groups that appears in Elek's paper \cite{Elek1998}, where an $\ell^p$-cohomology theory for graphs with uniformly bounded vertex degrees is worked out.
    Elek's cohomology is invariant under quasi-isometries of graphs, and gives our \Lcoh{} when computed on the Cayley graph of a finitely generated group.
    Thus, two quasi-isometric finitely generated groups have isomorphic \Lcoh{}, a result already proved by Gersten \cite[Theorem 11.4]{Gersten1992}.
    \begin{remark}
        \label{rem:finite_index}
        The special case of a group $G$ and a subgroup $H<G$ with finite index is particularly easy to understand.
        If $X$ is a cellular $K(G,1)$ space, then it has a finite cover $Y$ which is a $K(H,1)$ space.
        The spaces $X$ and $Y$ have the same universal cover $\Ucover{X}$, and a cellular cochain on $\Ucover{X}$ is bounded on $G$-orbits if and only if it is bounded on $H$-orbits.
        Therefore, $X$ and $Y$ have isomorphic \Lcoh{}.
    \end{remark}

    \section{Comparison maps and bounded cohomology}
    \label{sec:comparison}
    
    We start this section by quickly recalling the definition of bounded cohomology of groups and topological spaces.
    
    Let $G$ be a group and let $V$ be a normed $\R[G]$-module.
    The $\R[G]$-modules $C_b^k(G;V) = \{\alpha:G^{k+1}\to V\text{ bounded}\}$ form a subcomplex of the standard homogeneous complex $C^\bullet(G;V)$ associated to $G$ and $V$.
    The bounded cohomology $H_b^\bullet(G;V)$ is the cohomology of $C_b^\bullet(G;V)^G$, the complex of bounded $G$-invariant $V$-valued cochains.
    The inclusion $C_b^\bullet(G;V) \subseteq C^\bullet(G;V)$ induces a natural map $c^\bullet:H_b^\bullet(G;V) \to H^\bullet(G;V)$, called the comparison map.
    
    Let now $X$ be a connected topological space with fundamental group $G$ and universal cover $\Ucover{X}$.
    Then $C_b^k(\Ucover{X};V)$ denotes the $\R[G]$-module of bounded singular $V$-valued $k$-cochains in $\Ucover{X}$.
    The bounded cohomology $H_b^\bullet(X;V)$ is the cohomoloy of $C_b^\bullet(\Ucover{X};V)^G$.
    The inclusion of bounded cochains into singular cochains gives the comparison map $c^\bullet:H_b^\bullet(X;V) \to H^\bullet(X;V)$.
    
    The following theorem summarizes some results proved in \cite{Ivanov2017}.
    
    \begin{theorem}
        \label{thm:mapping-theorem}
        Let $X$ be a connected topological space with fundamental group $G = \pi_1(X,x_0)$.
        Let $V$ be a normed $\R[G]$-module.
        Assume that either $V$ is a dual normed $\R[G]$-module or the universal cover of $X$ is contractible (\ie, $X$ is a $K(G,1)$ space).
        Then $H_b^\bullet(G;V)$ and $H_b^\bullet(X;V)$ are canonically isomorphic.
    \end{theorem}
    
    Also $\Linf$-cohomology comes with a natural comparison map, defined as follows.
    Let $G$ be a group and $V$ a normed $\R[G]$-module.
    There is a natural inclusion $\iota: V \to \Linf(G,V)$ sending any vector $v \in V$ to the constant function $\iota(v):G\to V$ assuming the value $v$ at every element of $G$.
    The map $\iota$ is the only linear $G$-equivariant function such that $\iota(v)(1_G) = v$ for every $v \in V$.
    
    \begin{definition}
        Let $X$ be a connected topological space with fundamental group $G = \pi_1(X,x_0)$, and let $V$ be a normed $\R[G]$-module.
        We denote by $\iota^\bullet:H^\bullet(X;V) \to H_\Linfs^\bullet(X;V)$ the map induced by $\iota$ in cohomology.
        We also denote by $\iota^\bullet:H^\bullet(G;V) \to H_\Linfs^\bullet(G;V)$ the map induced by $\iota$ in group cohomology.
    \end{definition}
    
    We call $\iota^\bullet$ the comparison map from ordinary cohomology to \Lcoh{}.
    It can be thought of as the map induced by the inclusion of $G$-equivariant cochains into the set of cochains bounded on orbits.
    
    The following proposition and corollary establish a connection between \Lcoh{} and bounded cohomology.
    
    \begin{proposition}
        \label{prop:bounded_and_linf}
        Let $X, G$ and $V$ be as in \Cref{thm:mapping-theorem}.
        Then, for every $k \ge 1$, $H_b^k(X;\Linf(G,V)) = 0$.
    \end{proposition}
    \begin{proof}
        Let $\Ucover{X}$ be the universal cover of $X$.
        By definition, $H_b^\bullet(X;\Linf(G,V))$ is the cohomology of $C_b^\bullet(\Ucover{X};\Linf(G,V))^G$.
        As described in \Cref{sec:definition}, the evaluation at the identity of $G$ identifies $\Linf(G,V)$-valued cochains with $V$-valued (not necessarily $G$-equivariant) cochains.
        In particular, uniformly bounded $\Linf(G,V)$-valued cochains correspond to uniformly (not only orbit-wise) bounded $V$-valued cochains.
        Therefore, $H_b^\bullet(X;\Linf(G,V))$ is isomorphic to $H_b^\bullet(\Ucover{X};V)$.
        By \Cref{thm:mapping-theorem}, the bounded cohomology of $\Ucover{X}$ with coefficients in $V$ is trivial, since it is isomorphic to the bounded cohomology of the trivial group.
    \end{proof}
    
    \begin{corollary}
        \label{cor:null-composition}
        Let $X, G$ and $V$ be as in \Cref{thm:mapping-theorem}.
        Then the composition
        \[\begin{tikzcd}
            H_b^k(X;V) \ar[r,"c^k"] & H^k(X;V) \ar[r,"\iota^k"] & H_\Linfs^k(X;V)
        \end{tikzcd}\]
        is the zero map for every $k \ge 1$.
        In particular, for every group $G$, every normed $\R[G]$-module $V$ and every $k \ge 1$ the composition
        \[\begin{tikzcd}
            H_b^k(G;V) \ar[r,"c^k"] & H^k(G;V) \ar[r,"\iota^k"] & H_\Linfs^k(G;V)
        \end{tikzcd}\]
        is the zero map.
    \end{corollary}
    \begin{proof}
        The composition in the statement also factors as 
        \[\begin{tikzcd}
            H_b^k(X;V) \ar[r] & H_b^k(X;\Linf(G,V)) \ar[r,"c^k"] & H^k(X;\Linf(G,V)),
        \end{tikzcd}\]
        where the first map is induced by $\iota:V \to \Linf(G,V)$ and the second map is the comparison map from bounded to ordinary cohomology.
        By \Cref{prop:bounded_and_linf}, the central term is trivial.
    \end{proof}
    
    Other proofs for the second assertion of \Cref{cor:null-composition} can be found in \cite{Gersten1992}, \cite{Wienhard2012} and \cite{Blank2015}.
    In \cite{Kedra2009} there is a proof of the first assertion in the language of bounded de Rham cohomology: $X$ is a closed Riemannian manifold, $V = \R$, and the bounded de Rham cohomology of $\Ucover{X}$ takes the place of $H_\Linfs^k(X;V)$ as the last term of the sequence.

    \section{Characterization of amenable groups}
    \label{sec:amenability}
    
    We start by recalling the definition of amenable groups.
    
    \begin{definition}
        \label{def:amenable}
        Let $G$ be a group.
        A left-invariant mean on $G$ is a linear map $m: \Linf(G, \R) \to \R$ satisfying the following properties:
        \begin{enumerate}[label=(\arabic*)]
            \item \label{it:amenable-1} If $f_c:G\to \R$ is the constant function with value $c$, then $m(f_c) = c$;
            \item \label{it:amenable-2} If $f \in \Linf(G, \R)$ is non-negative, then $m(f) \ge 0$;
            \item \label{it:amenable-3} $m(g\cdot f) = m(f)$ for every $f \in \Linf(G, \R)$ and $g \in G$.
        \end{enumerate}
        The group $G$ is amenable if it admits a left-invariant mean.
    \end{definition}
    
    If $m: \Linf(G, \R) \to \R$ is a left-invariant mean, then it is continuous: properties \ref{it:amenable-1} and \ref{it:amenable-2} easily imply that $\inf f \le m(f) \le \sup f$ for any $f \in \Linf(G, \R)$, and in particular $\abs{m(f)} \le \norm{f}$.
    
    Let $W$ be a normed $\R[G]$-module, and let
    $V = W'= \{v:W \to \R\ \text{continuous}\allowbreak\ \text{linear}\allowbreak\ \text{map}\}$
    its topological dual, which is again a normed $\R[G]$-module.
    The action of $G$ on $V$ is the usual one:
    \begin{align*}
        (g\cdot v)(w) = v(g^{-1}\cdot w) & & \text{for every } g \in G, v \in V \text{ and } w \in W.
    \end{align*}
    
    If $G$ is amenable, the functional $\mu$ appearing in the following lemma can be interpreted as a $G$-equivariant averaging of bounded $V$-valued functions defined on $G$.
    
    \begin{lemma} \label{lemma:mean_on_dualspace}
        Let $G$ be an amenable group and $V = W'$ a dual normed $\R[G]$-module.
        Then there is a linear map $\mu: \Linf(G, V) \to V$ satisfying the following properties:
        \begin{enumerate}[label=(\arabic*)]
            \item \label{it:mean-1} If $f_v:G\to V$ is the constant function with value $v$, then $\mu(f_v) = v$;
            \item \label{it:mean-2} $\norm{\mu(f)} \le \norm{f}$ for every $f \in \Linf(G, V)$;
            \item \label{it:mean-3} $\mu(g\cdot f) = g \cdot \mu(f)$ for every $g \in G$ and $f \in \Linf(G, V)$.
        \end{enumerate}
        In other words, $\mu$ is a norm-nonincreasing and $G$-equivariant retraction for the inclusion of $V$ in $\Linf(G, V)$.
    \end{lemma}
    \begin{proof}
        For any $f \in \Linf(G,V)$ denote by $\Ev_f:W \to \Linf(G,\R)$ the function defined as $\Ev_f(w)(g) = f(g)(w)$ for every $w \in W$ and $g \in G$.
        The function $\Ev_f$ is linear and continuous, with operator norm $\norm{\Ev_f} \le \norm{f}$.
        Moreover, if $g \in G$ and $w \in W$ then
        \begin{align*}
             \Ev_{g\cdot f}(w)(h) &= (g\cdot f)(h)(w) = (g \cdot f(g^{-1}h))(w) = f(g^{-1}h)(g^{-1}w) \\
             &= \Ev_f (g^{-1}w)(g^{-1}h) = (g \cdot \Ev_f(g^{-1}w))(h)
        \end{align*}
        for every $h \in G$.
        That is, $\Ev_{g\cdot f}(w) = g \cdot \Ev_f(g^{-1}w)$.
        
        Let $m:\Linf(G,\R) \to \R$ be a left-invariant mean.
        The composition $m \circ \Ev_f:W \to \R$ is linear and continuous, so it gives an element of $V$.
        We define $\mu(f) = m \circ \Ev_f$.
        The function $\mu$ is linear and satisfies properties \ref{it:mean-1} and \ref{it:mean-2} of the statement.
        Moreover, if $f \in \Linf(G,V)$ and $g \in G$ then
        \begin{align*}
            \mu(g\cdot f)(w) &= m(\Ev_{g\cdot f}(w)) = m(g \cdot \Ev_f(g^{-1}w)) \\
            &= m(\Ev_f(g^{-1}w)) = \mu(f)(g^{-1}w) = (g\cdot \mu(f))(w)
        \end{align*}
        for every $w \in W$, and this proves property \ref{it:mean-3}.
    \end{proof}
    
    The following theorem gives the characterization of amenability already stated in the introduction.
    
    \begin{theorem} \label{thm:amenability_via_linfty_cohomology}
        For any group $G$ the following conditions are equivalent:
        \begin{enumerate}[label=(\roman*)]
            \item \label{cdt:1} $G$ is amenable;
            \item \label{cdt:2} The comparison map $\iota^1:H^1(G;V) \to H^1_\Linfs(G;V)$ is injective for every dual normed $\R[G]$-module $V$;
            \item \label{cdt:3} The comparison map $\iota^k:H^k(G;V) \to H^k_\Linfs(G;V)$ is injective for every dual normed $\R[G]$-module $V$ and every $k \in \N$.
        \end{enumerate}
    \end{theorem}
    \begin{proof}
        To prove the implication \ref{cdt:1} $\implies$ \ref{cdt:3}, fix a dual normed $\R[G]$-module $V$ and let $k \in \N$.
        \Cref{lemma:mean_on_dualspace} gives a map of $\R[G]$-modules $\mu:\Linf(G, V) \to V$ which is a left-inverse of the inclusion $\iota:V \to \Linf(G, V)$.
        The composition 
        \[\begin{tikzcd}
             H^k(G;V) \ar[r,"\iota^k"] & H_\Linfs^k(G;V) \ar[r,"\mu^k"] & H^k(G;V)
        \end{tikzcd}\]
        of the maps induced by $\iota$ and $\mu$ in cohomology is the identity.
        In particular, the first map must be injective, and this proves \ref{cdt:3}.
        
        To prove the implication \ref{cdt:2} $\implies$ \ref{cdt:1}, consider the dual $\R[G]$-module $V = (\Linf(G,\R)/\R)'$.
        Let $J \in C^1(G;V)$ be defined as
        $J(g_0, g_1)[f] = f(g_0) - f(g_1)$
        for every $g_0,g_1 \in G$ and $f \in \Linf(G,\R)$,
        where $[f]$ denotes the class of $f$ in the quotient $\Linf(G,\R)/\R$.
        It is easy to see that $J$ is a $G$-equivariant cocycle, so it defines a class $[J] \in H^1(G;V)$, called the Johnson class of $G$.
        
        We now prove that the comparison map sends $[J]$ to $0 \in H_\Linfs^1(G;V)$.
        Define a cochain $\alpha \in C^0(G;V)$ as follows:
        $\alpha(g)[f] = f(g) - f(1_G)$ for every $g \in G$ and $f \in \Linf(G,\R)$.
        Notice that $\abs{\alpha(g)[f]} \le 2 \cdot \norm{f}$.
        Taking the infimum as $f$ varies among the representatives of a class in $\Linf(G,\R)/\R$, we deduce that $\abs{\alpha(g)[f]} \le 2 \cdot \norm{[f]}$.
        Therefore, $\norm{\alpha(g)} \le 2$, and this means that $\alpha$ is bounded.
        In particular it is bounded on orbits.
        Moreover, it is clear that $\alpha$ is a primitive of $J$.
        Since $J$ has a primitive bounded on orbits, the comparison map sends $[J]$ to $0$, as we claimed.
        
        Now, condition \ref{cdt:2} implies that $[J] = 0 \in H^1(G;V)$, \ie, $J$ has a $G$-equivariant primitive $\beta \in C^0(G;V)^G$.
        Define $m:\Linf(G,\R) \to \R$ setting $m(f) = f(1_G) - \beta(1_G)[f]$ for every $f \in \Linf(G,\R)$.
        It is clear that $m$ is linear and continuous.
        Moreover, if $f_c$ is the constant function with value $c \in \R$, then $m(f_c) = c$.
        Let us check that $m$ is left-invariant.
        Fix $g \in G$ and $f \in \Linf(G,V)$.
        Then
        \begin{align*}
            m(g\cdot f)\ &=\ (g \cdot f)(1_G) - \beta(1_G)[g\cdot f] \\
            &=\ f(g^{-1}) - (g^{-1} \cdot \beta(1_G))[f] \\
            &=\ f(g^{-1}) - \beta(g^{-1})[f],
        \end{align*}
        where in the last equality we used the $G$-equivariance of $\beta$.
        Now
        \begin{align*}
            m(g\cdot f) - m(f)\ &=\ f(g^{-1}) - \beta(g^{-1})[f] - f(1_G) + \beta(1_G)[f] \\
            &=\ f(g^{-1}) - f(1_G) - J(1_G,g^{-1})[f]\ =\ 0.
        \end{align*}
        Therefore, $m$ is left-invariant.
        In general, $m$ does not satisfy property \ref{it:amenable-2} of \Cref{def:amenable}.
        However, properties \ref{it:amenable-1} and \ref{it:amenable-3} of \Cref{def:amenable}, together with the continuity of $m$, imply that $G$ is amenable anyway (see, \eg, \cite[Lemma 3.2]{Frigerio2017}).
        The remaining implication \ref{cdt:3} $\implies$ \ref{cdt:2} is clearly true.
    \end{proof}
    
    The statement of \Cref{thm:amenability_via_linfty_cohomology} is in some way dual to the following well-known characterization of amenable groups via bounded cohomology.
    For a proof of it, see, \eg, \cite{Frigerio2017}.
    \begin{theorem} \label{thm:amenability_via_bounded_cohomology}
        For any group $G$ the following conditions are equivalent:
        \begin{enumerate}[label=(\roman*)]
            \item $G$ is amenable;
            \item $H_b^1(G;V) = 0$ for every dual normed $\R[G]$-module $V$;
            \item $H_b^k(G;V) = 0$ for every dual normed $\R[G]$-module $V$ and every $k \ge 1$.
        \end{enumerate}
    \end{theorem}
    
    \begin{remark}
        Assuming \Cref{thm:amenability_via_bounded_cohomology}, the implication \ref{cdt:2} $\implies$ \ref{cdt:1} of \Cref{thm:amenability_via_linfty_cohomology} also follows very easily from the following general facts, which are true for any group $G$ and any normed $\R[G]$-module $V$:
        \begin{itemize}
            \item The composition $H^1_b(G;V) \to H^1(G;V) \to H^1_\Linfs(G;V)$ of the comparison maps is the zero map, by \Cref{cor:null-composition};
            \item The map $H^1_b(G;V) \to H^1(G;V)$ is injective.
        \end{itemize}
        Actually, the proofs of the implication \ref{cdt:2} $\implies$ \ref{cdt:1} in \Cref{thm:amenability_via_bounded_cohomology} and \Cref{thm:amenability_via_linfty_cohomology} follow very similar lines: in both cases the Johnson class in $H^1(G;V)$ is shown to be null (for different reasons) and the rest of the proof is identical.
        
    \end{remark}

    \section{Bounded primitives and linear isoperimetric inequalities}
    \label{sec:isoperimetric}
    
    Let $X$ be a connected \CW{} and let $\Ucover{X}$ be its universal cover, equipped with the cellular structure induced from $X$.
    As usual, the fundamental group $G = \pi_1(X,x_0)$ acts on $\Ucover{X}$ via covering transformations.
    In this section, chain and cochains will be cellular and real-valued.
    
    For every $k \ge 0$ we endow the vector space of cellular $k$-chains $C_k(\Ucover{X};\R)$ with the $\ell^1$-norm $\norm{\Ph}_1:C_k(\Ucover{X},\R) \to \R$, assigning to any chain the sum of the absolute values of its coefficients with respect to the basis of $k$-cells.
    
    The following theorem relates the vanishing of a class in $H_\Linfs^k(X;\R)$ to isoperimetric inequalities satisfied by its representatives.
    The proof is an application of the Hahn-Banach theorem.
    
    \begin{theorem}
        \label{thm:isoperimetric_cochain}
        Let $X$ be a connected \CW{} with universal cover $\Ucover{X}$.
        Let $\alpha \in C^k(\Ucover{X},\R)$ be a cellular cocycle.
        Suppose that there is a real number $\Lambda \ge 0$ such that
        \begin{equation}
            \label{eq:isoperimetric_cochain}
            \abs{\alpha(c)}\ \le\ \Lambda \cdot \norm{\Bd c}_1
        \end{equation}
        for every cellular chain $c \in C_k(\Ucover{X}, \R)$.
        Then $\alpha$ is bounded on orbits and $[\alpha] = 0 \in H_\Linfs^k(X;\R)$.
        Moreover, if the $(k-1)$-skeleton of $X$ is finite, then the converse implication also holds.
    \end{theorem}
    In the first part of the statement, the norm $\norm{\cdot}_1$ on $C_{k-1}(\Ucover{X};\R)$ may be replaced by any $\pi_1(X)$-invariant seminorm, and the following proof would remain valid without any changes.
    \begin{proof}
        The goal is to show that $\alpha$ admits a primitive bounded on orbits.
        Denote by $B_{k-1}(\Ucover{X};\R)$ the space of cellular $(k-1)$-boundaries in $\Ucover{X}$.
        Define a linear function $\mu:B_{k-1}(\Ucover{X},\R) \to \R$, setting $\mu(\Bd c) = \alpha(c)$ for every $c \in C_k(\Ucover{X},\R)$.
        The function $\mu$ is well defined: if $c$ and $c'$ are two cellular chains such that $\Bd c = \Bd c'$, then \eqref{eq:isoperimetric_cochain} implies that
        \[ \abs{\alpha(c) - \alpha(c')}\ =\ \abs{\alpha(c-c')}\ \le\  \Lambda \cdot \norm{\Bd(c-c')}_1\ =\ 0, \]
        so $\alpha(c) = \alpha(c')$.
        Moreover, $\abs{\mu(b)} \le \Lambda \cdot \norm{b}_1$ for every cellular boundary $b \in B_{k-1}(\Ucover{X},\R)$.
        The Hahn-Banach theorem allows us to extend $\mu$ to the whole space of cellular chains $C_{k-1}(\Ucover{X},\R)$ in such a way that
        \begin{equation}
            \label{eq:bounded_norm_primitive}
            \abs{\mu(x)}\ \le\ \Lambda \cdot \norm{x}_1
        \end{equation}
        for every $x \in C_{k-1}(\Ucover{X},\R)$.
        We have $\Cbd\mu = \alpha$ by construction.
        The fact that $\mu$ is bounded on orbits (actually, it is globally bounded) follows from \eqref{eq:bounded_norm_primitive}.
        
        For the opposite direction, suppose that the $(k-1)$-skeleton of $X$ is finite and that $\alpha = \Cbd \mu$ with $\mu \in C^{k-1}(\Ucover{X},\R)$ bounded on orbits.
        Each orbit of $(k-1)$-cells in $\Ucover{X}$ corresponds to a $(k-1)$-cell in $X$, so they are in a finite number.
        Therefore, the values that $\mu$ assigns to the cells of $\Ucover{X}$ are bounded in absolute value by a real number $\Lambda \ge 0$.
        This implies the inequality
        $\abs{\mu(x)} \le \Lambda \cdot \norm{x}_1$
        for every cellular chain $x \in C_{k-1}(\Ucover{X},\R)$.
        Now
        $\abs{\alpha(c)} = \abs{\mu(\Bd c)} \le \Lambda \cdot \norm{\Bd c}_1$
        for every $c \in C_k(\Ucover{X},\R)$, and the proof is complete.
    \end{proof}
    
    The case in which the $k$-skeleton of $X$ is finite (which is not sufficient for the application by Frigerio and Sisto mentioned in the introduction) is a particular instance of a result by Manin \cite[Theorem 5.1 and Corollary 5.2]{Manin2016}, who considers cochains with coefficients in a finite-dimensional vector space endowed with a (possibly) different polyhedral norm for every $(k-1)$-cell of $\Ucover{X}$, and the isoperimetric inequality as well as the boundedness of cochains are defined accordingly.
    
    We now describe some applications of \Cref{thm:isoperimetric_cochain}.
    
    \begin{example}
        Let $G$ be a hyperbolic group and let $X$ be a $K(G,1)$ with finite $2$-skeleton.
        Then $\Ucover{X}$ satisfies a linear isoperimetric inequality for cellular $1$-boundaries (see \cite[Theorem A.1]{Gersten1998}): there is a real number $L \ge 0$ such that $\inf\{\norm{c}_1 \mid \Bd c = b\} \le L \cdot \norm{b}_1$ for every $b \in B_1(\Ucover{X};\R)$.
        This easily implies that every cocycle $\alpha \in C_\Linfs^2(\Ucover{X};\R)$ satisfies \eqref{eq:isoperimetric_cochain} for some $\Lambda \ge 0$.
        Therefore, $H_\Linfs^2(G;\R) = H_\Linfs^2(X;\R) = 0$.
        This is a particular case of \Cref{thm:intro_linf_characterizes_hyperbolic} and was established by Gersten in \cite{Gersten1998,Gersten1996b}.
    \end{example}
    
    \begin{example}
        \label{example:manifold_amenable}
        Let $M$ be a closed connected manifold of dimension $n$, with a fixed triangulation.
        Then $H_\Linfs^n(M,\R) \neq 0$ if and only if the fundamental group of $M$ is amenable.
        Variations of this result are proved in \cite[Remark following Theorem 3.1]{BW1992}, \cite{Sikorav2001} and \cite[Theorem 3.4]{BK2008}; we sketch a simple proof using \Cref{thm:isoperimetric_cochain}.
        
        If $H_\Linfs^n(M,\R) \neq 0$, then the quantity
        \[\inf\left\{\frac{\norm{\Bd{c}}_1}{\norm{c}_1}\ \middle|\ c\text{ is a nontrivial cellular }n\text{-chain in }\Ucover{M}\right\}\]
        must be $0$, otherwise every cocycle $\alpha \in C_\Linfs^n(\Ucover{M};\R)$ would satisfy the assumptions of \Cref{thm:isoperimetric_cochain} and the \Lcoh{} in degree $n$ would vanish.
        This easily implies the F\o{}lner condition for (a Cayley graph of) the fundamental group of $M$ \cite[Definition 18.2]{DK2018}, which is equivalent to the amenability of the fundamental group.
        
        On the other hand, if the fundamental group is amenable, then the comparison map $H^n(M;\R) \to H_\Linfs^n(M;\R)$ is injective; this follows by the same argument used in the proof of the implication \ref{cdt:1} $\implies$ \ref{cdt:2} of \Cref{thm:amenability_via_linfty_cohomology}.
        In particular, if $M$ is orientable, then $H_\Linfs^n(M;\R) \neq 0$.
        If $M$ is not orientable, we reach the same conclusion by passing to the oriented double cover, whose \Lcoh{} is naturally isomorphic to that of $M$.
    \end{example}
    
    \begin{example}
        As mentioned in the introduction, Frigerio and Sisto \cite{Frigerio2020} have found a finitely generated group $G$ such that the sequence $H_b^2(G;\R) \to H^2(G;\R) \to H_\Linfs^2(G;\R)$ is not exact.
        They define $G$ by means of an explicit presentation, using a finite number of generators and an infinite sequence of relations $(r_0, r_1, \ldots, r_n, \ldots)$.
        Then they show that this presentation satisfies a weighted linear isoperimetric inequality, with the relation $r_n$ having weight $2n+1$.
        That is, any word $w$ representing the identity can be written as
        \[ w\ =\ \prod_{i=1}^l v_i r_{n_i} v_i^{-1} \]
        for some words $v_i$ and indices $n_i$ with $(2n_1+1) + \cdots + (2n_l+1) \le \abs{w}$, where $\abs{w}$ denotes the length of $w$.
        Moreover, they show that the presentation is aspherical: the $2$-dimensional complex $X$ associated to the presentation is a $K(G,1)$.
        From the weighted isoperimetric inequality of the presentation, it is easy to obtain a weighted linear isoperimetric inequality for $1$-boundaries in $\Ucover{X}$: $\inf\{\norm{c}_\mathrm{w} \mid \Bd c = b\} \le \norm{b}_1$ for every $b \in B_1(\Ucover{X};\R)$, where $\norm{\Ph}_\mathrm{w}$ is the weighted $\ell^1$-norm assigning weight $2n+1$ to the $2$-cells corresponding to the relation $r_n$.
        It is immediate to construct $G$-invariant cocycles $\alpha \in C^2(\Ucover{X};\R)^G$ such that $\abs{\alpha(c)} \le \norm{c}_\mathrm{w}$ for every $c \in C_2(\Ucover{X};\R)$.
        These cocycles define classes in $H^2(G;\R)$ that in \cite{Frigerio2020} are called ``slow'' classes.
        Any such cocycle $\alpha$ satisfies the inequality \eqref{eq:isoperimetric_cochain} with $\Lambda = 1$.
        Recall that, at the cochain level, the comparison map $\iota^2:H^2(X;\R) \to H_\Linfs^2(X;\R)$ is given by the inclusion of $G$-invariant cochains into the space of cochains bounded on orbits.
        In particular the class $\iota^2([\alpha]) \in H_\Linfs^2(X;\R)$ is represented by $\alpha$ itself, and by \Cref{thm:isoperimetric_cochain} this class is the trivial one.
        Frigerio and Sisto show that there are plenty of slow classes that are not bounded, \ie, they do not lie in the image of $c^2:H_b^2(G;\R) \to H^2(G;\R)$.
        Therefore, the sequence $H_b^2(G;\R) \to H^2(G;\R) \to H_\Linfs^2(G;\R)$ is not exact.
    \end{example}

    \section{Characterization of relative hyperbolicity}
    \label{sec:relative_hyperbolicity}
    
    Let $G$ be a group and let $\mathcal{H} = \{H_i\}_{i \in I}$ be a collection of subgroups of $G$.
    We do not exclude repetitions, \ie, we allow $H_i = H_j$ for $i \neq j$.
    In such situations we say that $(G,\mathcal{H})$ is a group pair.
    If $V$ is an $\R[G]$-module, we denote by $H^\bullet(G,\mathcal{H};V)$ the cohomology of $G$ relative to $\mathcal{H}$ with coefficients in $V$, as defined by Bieri and Eckmann in \cite{BE1978}.
    For our purposes, it is not necessary to recall the algebraic definition; it is enough to know the following topological description.
    \begin{proposition}[{\cite{BE1978}}]
        Let $(G,\mathcal{H})$ be a group pair, with $\mathcal{H} = \{H_i\}_{i \in I}$.
        Let $(X,Y)$ be a classifying pair for $(G,\mathcal{H})$, \ie, a pair of topological spaces such that:
        \begin{itemize}
            \item $X$ is a cellular $K(G,1)$ space;
            \item $Y$ is a subcomplex of $X$ and has connected components $\{Y_i\}_{i \in I}$, with $Y_i \neq Y_j$ for indices $i \neq j \in I$;
            \item For every $i\in I$, $Y_i$ is a $K(H_i,1)$ space and the inclusion in $X$ induces, with respect to some (fixed, but omitted from the notation) path joining the respective basepoints, the inclusion of $H_i$ in $G$.
        \end{itemize}
        Let $V$ be an $\R[G]$-module.
        Then $H^\bullet(G,\mathcal{H};V)$ is canonically isomorphic to $H^\bullet(X,Y;V)$, where the latter is the usual cohomology of pairs of spaces.
    \end{proposition}
    
    Now, if $V$ is a normed $\R[G]$-module, we define the relative \Lcoh{} as $H_\Linfs^\bullet(G,\mathcal{H};V) = H^\bullet(G,\mathcal{H};\Linf(G,V))$.
    
    Let $(X,Y)$ be a classifying pair for $(G,\mathcal{H})$.
    Let $(\Ucover{X},\Ucover{Y})$ be the universal cover of $(X,Y)$, meaning that $\Ucover{X}$ is the universal cover of $X$ and $\Ucover{Y}$ is the preimage of $Y$ in $\Ucover{X}$.
    The complex $C^\bullet(\Ucover{X},\Ucover{Y};V)$ of relative cochains, \ie, cochains vanishing on $\Ucover{Y}$, has a subcomplex $C_\Linfs^\bullet(\Ucover{X},\Ucover{Y};V)$ of cochains bounded on orbits, defined by adapting \Cref{defn:bounded_on_orbits} in the obvious way.
    Then $H_\Linfs^\bullet(G,\mathcal{H};V)$ is canonically isomorphic to the cohomology of this subcomplex.
    
    There is also a relative version of bounded cohomology, introduced by Mineyev and Yaman \cite{MY} (see also \cite{Bla2016, Fra2018}).
    Its topological description is the following: if $V$ is a normed $\R[G]$-module and $(X,Y)$ is a classifying pair for $(G,\mathcal{H})$, then $H_b^\bullet(G,\mathcal{H};V)$ is canonically isomorphic to $H_b^\bullet(X,Y;V)$, where the latter is the cohomology of the complex of $G$-equivariant bounded singular cochains on $\Ucover{X}$ vanishing on $\Ucover{Y}$.
    As in the non-relative setting, the inclusion of bounded cochains into unbounded ones gives the comparison map $c^\bullet:H_b^\bullet(G,\mathcal{H};V) \to H^\bullet(G,\mathcal{H};V)$.
    
    We need the following relative version of \Cref{prop:bounded_and_linf}:
    \begin{proposition}
        \label{prop:bounded_and_linf_relative}
        Let $(G,\mathcal{H})$ be a group pair and $V$ a normed $\R[G]$-module.
        Then $H_b^k(G,\mathcal{H};\Linf(G,V)) = 0$ for every $k \ge 2$.
    \end{proposition}
    \begin{proof}
        Fix a classifying pair $(X,Y)$ with universal cover $(\Ucover{X},\Ucover{Y})$.
        Then, by the same argument used in the proof of \Cref{prop:bounded_and_linf}, $H_b^\bullet(G,\mathcal{H};\Linf(G,V))$ is isomorphic to the cohomology of $C_b^\bullet(\Ucover{X},\Ucover{Y};V)$.
        
        Let $\alpha \in C_b^k(\Ucover{X},\Ucover{Y};V)$ be a cocycle, with $k \ge 2$.
        Since $\Ucover{X}$ is contractible, $\alpha$ has a bounded primitive $\beta$, which however might be nonzero on $\Ucover{Y}$.
        For example, $\beta$ can be obtained by ``integrating along cones'': using a contraction of $\Ucover{X}$ it is easy to  associate to any singular $(k-1)$-simplex $\sigma$ a singular $k$-simplex $c(\sigma)$, thought of as a cone over $\sigma$, in a way compatible with faces; then $\beta(\sigma) = \alpha(c(\sigma))$ is a primitive with norm not greater than that of $\alpha$.
        
        The components of $\Ucover{Y}$ are contractible because the inclusion $Y \subseteq X$ is $\pi_1$-injective and the components of $Y$ are aspherical.
        Hence, the restriction of $\beta$ to any component of $\Ucover{Y}$ (which is a cocycle, since $\alpha$ vanishes on $\Ucover{Y}$) has a bounded primitive, with norm not greater than that of $\beta$ (what is important is that there is a uniform bound independent of the component).
        Here we are using that $k \ge 2$.
        Let $\gamma\in C_b^{k-2}(\Ucover{X};V)$ be the cochain that assigns value $0$ to simplices not contained in $\Ucover{Y}$, and that on the components of $\Ucover{Y}$ coincides with the bounded primitives of the restrictions of $\beta$ considered above.
        Now, $\beta - \Cbd{\gamma}$ is a bounded primitive of $\alpha$ vanishing on $\Ucover{Y}$.
    \end{proof}
    
    \begin{corollary}
        \label{cor:null-composition_relative}
        Let $(G,\mathcal{H})$ be a group pair and $V$ a normed $\R[G]$-module.
        Denote by $\iota^\bullet$ the comparison map from ordinary to \Lcoh{} induced by the inclusion of $V$ into $\Linf(G,V)$.
        Then the composition
        \[\begin{tikzcd}
            H_b^k(G,\mathcal{H};V) \ar[r,"c^k"] & H^k(G,\mathcal{H};V) \ar[r,"\iota^k"] & H_\Linfs^k(G,\mathcal{H};V)
        \end{tikzcd}\]
        is the zero map for every $k \ge 2$.
    \end{corollary}
    \begin{proof}
        Identical to the proof of \Cref{cor:null-composition}.
    \end{proof}
    
    We now revisit the notion of \emph{strong vanishing} of $H_\Linfs^n(G,\mathcal{H};V)$, which has been considered by Gersten in the absolute case \cite{Gersten1996a, Gersten1998} and by Gautero and Heusener in the relative case \cite{GH2009}.
    
    Let $(X,Y)$ be a classifying pair for $(G,\mathcal{H})$, with universal cover $(\Ucover{X},\Ucover{Y})$.
    We say that $(X,Y)$ is relatively finite in dimension $\le n$ if $X$ has only a finite number of cells of dimension $\le n$ outside $Y$.
    For any cellular cochain $\alpha \in C^k(\Ucover{X},\Ucover{Y};V)$ define $\norm{\alpha}_\infty = \sup\{\norm{\alpha(c)} \mid c \text{ is a }k\text{-cell}\}.$
    Assuming that $(X,Y)$ is relatively finite in dimension $\le n$, and that $\alpha$ is a $k$-cochain with $k\le n$, then it is clear that $\norm{\alpha}_\infty < \infty$ if and only if $\alpha$ is bounded on orbits.
    
    \begin{definition}
        Let $(G,\mathcal{H})$ be a group pair, $V$ a normed $\R[G]$-module and $n \ge 0$ an integer.
        We say that $H_\Linfs^n(G,\mathcal{H};V)$ vanishes \emph{strongly} if:
        \begin{itemize}
            \item $(G,\mathcal{H})$ has at least a classifying pair $(X,Y)$ which is relatively finite in dimension $\le n$;
            \item For every such pair, there is a constant $L > 0$ such that every cellular $n$-cocycle $\alpha\in C_\Linfs^n(\Ucover{X},\Ucover{Y};V)$ has a primitive $\beta$ with $\norm{\beta}_\infty \le L \cdot \norm{\alpha}_\infty$.
        \end{itemize}
    \end{definition}
    
    If the second condition holds for at least one classifying pair which is relatively finite in dimension $\le n$, then it holds for all of them (possibly with a different constant $L$).
    This follows immediately from the following straightforward generalization of \cite[Proposition 4.6]{Gersten1996b}.
    
    \begin{proposition}
        \label{prop:strong_vanishing}
        Let $(G,\mathcal{H})$ be a group pair admitting a classifying pair $(X,Y)$ relatively finite in dimension $\le n$.
        Let $V$ be a normed $\R[G]$-module.
        Then $H_\Linfs^n(G,\mathcal{H};V)$ vanishes strongly if and only if $H_\Linfs^n(G,\mathcal{H};\Linf(\N,V)) = 0$, where $\Linf(\N,V)$ is the $\R[G]$-module of bounded sequences in $V$, endowed with the sup norm.
    \end{proposition}
    \begin{proof}
        All the cochains considered in this proof are cellular.
        There is an obvious correspondence between $C_\Linfs^n(\Ucover{X},\Ucover{Y};\Linf(\N,V))$ and sequences of equibounded (with respect to $\norm{\cdot}_\infty$) cochains in $C_\Linfs^n(\Ucover{X},\Ucover{Y};V)$.
        
        If $H_\Linfs^n(G,\mathcal{H};V)$ does not vanish strongly, then there is a sequence of cocycles $\alpha_m \in C_\Linfs^n(\Ucover{X},\Ucover{Y};V)$, indexed by $m \in \N$, with $\norm{\alpha_m}_\infty \le 1$ but without a primitive of norm $\le m$.
        This sequence corresponds to a certain cocycle $\alpha \in C_\Linfs^n(\Ucover{X},\Ucover{Y};\Linf(\N,V))$.
        Then $\alpha = \Cbd\beta$, and $\beta$ gives a primitive $\beta_n$ to each $\alpha_n$ with norm $\norm{\beta_n}_\infty \le \norm{\beta}_\infty$, in contrast with the definition of $\alpha_n$.
        A similar reasoning applies for the converse implication.
    \end{proof}
    
    There are several equivalent ways to define relatively hyperbolic group pairs.
    Among them, we use the characterization introduced by Groves and Manning in \cite{GM2008} that exploits the properties of the \emph{cusped space} associated to $(G,\mathcal{H})$.
    This approach requires $G$ to be finitely generated; note that there are more general treatments of relative hyperbolicity that do not need this assumption.
    For an account of these, we refer to \cite{Hru2010} and \cite{Osi2006}.
    
    Let $(G,\mathcal{H})$ be a group pair, with $\mathcal{H} = \{H_i\}_{i\in I}$.
    A relative generating set for $(G,\mathcal{H})$ is a set $S$ with a function $S \to G$ such that the induced homomorphism $\left(\ast_{i\in I} H_i\right) \ast F_S \to G$ is surjective, where $F_S$ denotes the free group on $S$.
    A relative presentation for $(G,\mathcal{H})$ is given by a relative generating set $S$ and a set of relations $R \subseteq \left(\ast_{i\in I} H_i\right) \ast F_S$ whose normal closure is equal to the kernel of the homomorphism onto $G$.
    The relative presentation is finite if both $S$ and $R$ are finite.
    Hereafter, we assume that $G$ is finitely generated and that
    $(G,\mathcal{H})$ admits a finite relative presentation.
    By \cite[Theorem 1.1]{Osi2006}, these assumptions imply that $\mathcal{H} = \{H_1, \ldots, H_m\}$ is finite and that each $H_i$ is finitely generated.
    
    We refer to \cite{GM2008} for the precise construction of the cusped space $\mathcal{C}$ associated to a finite relative presentation (it is obtained from the relative Cayley complex by attaching a \emph{combinatorial horoball} to every coset of the form $gH_i$).
    Here, we just list its salient properties:
    \begin{enumerate}[label=(\arabic*)]
        \item $\mathcal{C}$ is a simply connected, locally compact $2$-dimensional \CW{};
        \item \label{p:horoballs} $\mathcal{C}$ contains a distinguished collection of disjoint subcomplexes, called \emph{horoballs}, and each horoball is simply connected;
        \item \label{p:action} $G$ acts freely and properly discontinuously on $\mathcal{C}$, preserving the cellular structure and the union of the horoballs;
        \item \label{p:rel_cocompact} The action is not cocompact (unless $\mathcal{H}$ is empty), but there are only finitely many orbits of cells not lying inside the horoballs;
        \item \label{p:cosets} There is a natural $G$-equivariant bijection between the set of horoballs and the set of cosets of the form $gH_i$, with $g \in G$ and $H_i \in \mathcal{H}$;
        \item \label{p:uniform_boundary} There is a uniform bound on the norm $\norm{\Bd c}_1$ of the boundary of a $2$-cell $c$.
        Here, $\norm{\cdot}_1$ is the norm we already used in \Cref{sec:isoperimetric}.
        \item \label{def:relative_hyperbolicity} $(G,\mathcal{H})$ is relatively hyperbolic if and only if $\mathcal{C}$ satisfies a linear isoperimetric inequality for cellular $1$-cycles, \ie, there is a constant $L > 0$ such that every $1$-cycle $b$ is the boundary of a $2$-chain $c$ with $\norm{c}_1 \le L \cdot \norm{b}_1$;
        \item \label{p:horoball_3} Every horoball satisfies a linear isoperimetric inequality for cellular $1$-cycles, with constant $L = 3$.
    \end{enumerate}
    We finally have all the ingredients to state and prove our characterization.
    
    \begin{theorem}
        \label{thm:relative_hyperbolicity}
        Let $(G,\mathcal{H})$ be a group pair admitting a finite relative presentation, with $G$ finitely generated.
        Then the following conditions are equivalent:
        \begin{enumerate}[label=(\roman*)]
            \item \label{it:rel_hyp} The pair $(G,\mathcal{H})$ is relatively hyperbolic;
            \item \label{it:comp_surj} The comparison map $c^2:H^2_b(G,\mathcal{H};V) \to H^2(G,\mathcal{H};V)$ is surjective for every normed $\R[G]$-module $V$;
            \item \label{it:linf_null} $H_\Linfs^2(G,\mathcal{H};V) = 0$ for every normed vector space $V$;
            \item \label{it:strong_null} $H_\Linfs^2(G,\mathcal{H};\R)$ vanishes strongly.
        \end{enumerate}
        Moreover, if $(G,\mathcal{H})$ is relatively hyperbolic, then for every $k \ge 2$:
        \begin{enumerate}[label=(\roman*-$k$)]
            \setcounter{enumi}{1}
            \item \label{it:comp_surj_k} The comparison map $c^k:H^k_b(G,\mathcal{H};V) \to H^k(G,\mathcal{H};V)$ is surjective for every normed $\R[G]$-module $V$;
            \item \label{it:linf_null_k} $H_\Linfs^k(G,\mathcal{H};V) = 0$ for every normed vector space $V$;
            \item \label{it:strong_null_k} $H_\Linfs^k(G,\mathcal{H};\R)$ vanishes strongly.
        \end{enumerate}
    \end{theorem}
    \begin{proof}
        Throughout the proof, $k$ is an integer not smaller than $2$.
        
        The implications \ref{it:rel_hyp} $\iff$ \ref{it:comp_surj} $\implies$ \ref{it:comp_surj_k}, which constitute a characterization of relative hyperbolicity using bounded cohomology, are the main result of \cite{Fra2018}, and we do not discuss them further.
        
        To prove \ref{it:comp_surj_k} $\implies$ \ref{it:linf_null_k}, take $W = \Linf(G,V)$.
        If \ref{it:comp_surj_k} holds then the map $c^k:H^k_b(G,\mathcal{H};W) \to H^k(G,\mathcal{H};W)$ is surjective.
        But the domain of this map is $0$ by \Cref{prop:bounded_and_linf_relative}, and its target is $H_\Linfs^k(G,\mathcal{H};V)$.
        In particular, this establishes \ref{it:comp_surj} $\implies$ \ref{it:linf_null}.
        
        Suppose now that \ref{it:linf_null} holds.
        Then \ref{it:strong_null} immediately follows from \Cref{prop:strong_vanishing}, provided that there exists a classifying pair $(X,Y)$ which is relatively finite in dimension $\le 2$.
        But this is true for every pair $(G,\mathcal{H})$ admitting a finite relative presentation.
        
        To close the circle of equivalences, we prove \ref{it:rel_hyp} assuming \ref{it:strong_null}.
        Let $\mathcal{C}$ be a cusped space associated to $(G,\mathcal{H})$; the goal is to show that it satisfies a linear isoperimetric inequality for cellular $1$-cycles.
        Let $\mathcal{C} \to X'$ be the quotient map induced by the action of $G$, and let $Y'$ be the image in $X'$ of the union of the horoballs.
        If $X'$ and $Y'$ were aspherical, they would constitute a classifying pair for $(G,\mathcal{H})$.
        In any case, we obtain a classifying pair $(X,Y)$ by attaching to $(X',Y')$ cells of dimension $\ge 3$ (first, add cells to the subcomplex $Y$, making its components aspherical without altering fundamental groups; then, attach cells to $X$ to make it aspherical).
        Let $(\Ucover{X},\Ucover{Y})$ be its universal cover.
        The $2$-skeleton of $\Ucover{X}$ is $\mathcal{C}$, and the $2$-skeleton of $\Ucover{Y}$ is the union of the horoballs.
        By \ref{it:strong_null}, there is a constant $L > 0$ such that every cellular cocycle $\alpha \in C_\Linfs^2(\Ucover{X},\Ucover{Y};\R)$ has a primitive $\beta$ with $\norm{\beta}_\infty \le L\cdot\norm{\alpha}_\infty$.
        
        We endow the vector space $B_1(\Ucover{X};\R)$ of cellular $1$-cycles with the \emph{relative filling seminorm} $\norm{b}_f = \inf\{\norm{c}_1 \mid c \in C_2(\Ucover{X};\R), b - \Bd{c} \subseteq \Ucover{Y}\}$.
        
        Fix $b \in B_1(\Ucover{X};\R)$, and suppose that $b$ is not entirely contained in $\Ucover{Y}$, so that $\norm{b}_f > 0$.
        A standard application of the Hahn-Banach theorem gives a linear function $h:B_1(\Ucover{X};\R) \to \R$ of norm $1$ with respect to $\norm{\cdot}_f$ and such that $h(b) = \norm{b}_f$.
        Define $\alpha \in C_\Linfs^2(\Ucover{X},\Ucover{Y};\R)$ such that $\alpha(c) = h(\Bd{c})$ for every $2$-chain $c$.
        Notice that $\alpha$ is a cocycle, and it vanishes on $\Ucover{Y}$ because $h$ does.
        Moreover, $\abs{\alpha(c)} \le \norm{\Bd{c}}_f \le \norm{c}_1$, therefore $\norm{\alpha}_\infty \le 1$.
        Let $\beta \in C_\Linfs^1(\Ucover{X},\Ucover{Y};\R)$ be a primitive of $\alpha$ with $\norm{\beta}_\infty \le L$.
        By construction, $\norm{b}_f = h(b) = \beta(b) \le L \cdot \norm{b}_1$.
        Hence, there is a $2$-chain $c$ with $\norm{c}_1 \le (L+1)\cdot \norm{b}_1$ such that $b-\Bd{c}$ is supported on the horoballs.
        By property \ref{p:horoball_3} of the cusped space, $b-\Bd{c} = \Bd{e}$ for some $2$-cycle $e$ with $\norm{e}_1 \le 3 \cdot \norm{b - \Bd{c}}_1$.
        Property \ref{p:uniform_boundary} of the cusped space implies that $\norm{\Bd{c}}_1 \le K \cdot \norm{c}_1$ for some constant $K > 0$ independent of $c$.
        Therefore, $b = \Bd{(c + e)}$ and $\norm{c + e}_1 \le ((L+1)(3K+1)+3)\cdot\norm{b}_1$.
        This argument is valid for any $1$-cycle $b$, proving that $\mathcal{C}$ satisfies a linear isoperimetric inequality for cellular $1$-cycles and that $(G,\mathcal{H})$ is relatively hyperbolic.
        
        It remains to prove \ref{it:strong_null_k} assuming \ref{it:rel_hyp}.
        We already know that \ref{it:rel_hyp} implies \ref{it:linf_null_k}, and the same argument used for the implication \ref{it:linf_null} $\implies$ \ref{it:strong_null} also applies for arbitrary $k \ge 2$, except that we additionally need to prove the existence of a classifying pair which is relatively finite in dimension $\le k$.
        By considering Rips complexes over the $1$-skeleton of $\mathcal{C}$, under the assumption of relative hyperbolicity, Franceschini in \cite[Corollary 4.8]{Fra2018} obtains a locally finite simplicial complex $\mathcal{R}$ enjoying the same properties \ref{p:horoballs}--\ref{p:cosets} of the cusped space, with the following differences:
        \begin{itemize}
            \item $\mathcal{R}$ and the horoballs are contractible; 
            \item The action of $G$ is not necessarily free.
        \end{itemize}
        If $G$ is torsion-free then the action is free, and the quotient of $\mathcal{R}$ with respect to the action gives a classifying pair which is relatively finite in every dimension.
        In the presence of torsion, cells of $\mathcal{R}$ might have finite nontrivial stabilizers.
        In this case, the Borel construction described in \cite[Chapter 6, Theorem 7.3.1 and Remark 7.3.2]{Geo2007} can be used to construct the desired classifying pair anyway.
    \end{proof}

    \section{Some undecidable algorithmic problems}
    \label{sec:uncomputability}
    
    In this section we show that there is no algorithm to compute the \Lcoh{} of a group or space.
    More precisely, for every $n \ge 1$ the following algorithmic problems are undecidable:
    \begin{itemize}
        \item Given a finite presentation $\langle S|R\rangle$, decide whether $H_\Linfs^n(\langle S|R\rangle;\R) = 0$;
        \item Given a finite simplicial complex $X$, decide whether $H_\Linfs^n(X;\R) = 0$.
    \end{itemize}
    Here, a simplicial complex is described by the set of its vertices and the collection of subsets of vertices constituting the simplices.
    Actually, in \Cref{thm:undecidable_complexes} below we consider a distinct algorithmic problem for every fixed dimension $m$ of the simplicial complex accepted as input, and we prove that for many pairs $(n,m)$ the problem is undecidable.
    
    An important source of undecidable problems concerning finite presentations of groups is the Adian-Rabin theorem \cite[Theorem 12.32]{Rot1995}: if $\mathcal{P}$ is a class of groups such that
    \begin{itemize}
        \item $\mathcal{P}$ is closed under isomorphisms,
        \item $\mathcal{P}$ contains at least a finitely presentable group, and
        \item there is a finitely presentable group $H$ that does \emph{not} appear as a subgroup of a group in $\mathcal{P}$,
    \end{itemize}
    then there is no algorithm to decide whether $\langle S|R\rangle \in \mathcal{P}$.
    Classes $\mathcal{P}$ satisfying the above conditions are called \emph{Markov classes}.
    For instance, finite (respectively, amenable) groups constitute a Markov class: any infinite (respectively, non-amenable) finitely presentable group $H$ satisfies the third condition.
    
    A key component of our arguments is the following result.
    We will use it to reduce the algorithmic problem of deciding whether a group is finite to a problem concerning the vanishing of \Lcoh{}.
    
    \begin{proposition}
        \label{prop:finiteness}
        Let $X$ be a connected \CW{} with finite $1$-skeleton and infinite fundamental group.
        Then $H_\Linfs^{n+1}\left((S^1)^n\times X;\R\right) \neq 0$ for every $n \in \N$.
    \end{proposition}
    \begin{proof}
        We address the case $n = 0$ first.
        Fix a vertex $x_0 \in \Ucover{X}$ in the universal cover of $X$.
        For any vertex $v \in \Ucover{X}$, define $u(v) \in \N$ to be the minimum number of $1$-cells in paths from $x_0$ to $v$.
        This quantity is well defined because $\Ucover{X}$ is connected.
        Moreover, it attains arbitrarily large values, because only finitely many vertices can be reached from $x_0$ by paths of a fixed length (since $X$ has a finite $1$-skeleton), but there are infinitely many vertices (since the fundamental group of $X$ is infinite).
        The function $u$ defines a cellular $0$-cochain.
        Let $\alpha = \Cbd{u}$.
        Then $\alpha$ is bounded: it assigns the value $\pm 1$ to every $1$-cell.
        It defines a nontrivial class in $H_\Linfs^1(X;\R)$, because every primitive of $\alpha$ is obtained from $u$ by adding a constant, and so is unbounded.
        
        Now we consider $n > 0$.
        The universal cover of $(S^1)^n\times X$ is $\R^n \times \Ucover{X}$, where $\R^n$ has the cellular structure given by the regular grid of unit $n$-cubes.
        Define $\beta \in C_\Linfs^{n+1}(\R^n \times \Ucover{X};\R)$ by setting $\beta(e^n \times f^1) = \alpha(f^1)$ for every $n$-cell $e^n$ of $\R^n$ and every $1$-cell $f^1$ of $\Ucover{X}$, and assigning value $0$ to cells not of this form.
        Using that $\R^n$ has no $(n+1)$-cells and that $\alpha$ is a cocycle, it is easy to check that $\beta$ is a cocycle.
        We now prove that no primitive of $\beta$ is bounded on orbits.
        Fix an integer $k > 0$.
        Let $p$ be a $1$-chain in $\Ucover{X}$ corresponding to a shortest path from $x_0$ to a vertex having distance $k$ from $x_0$.
        Then $\abs{\alpha(p)} = \norm{p}_1 = k$.
        Let $q$ be the $n$-chain in $\R^n$ defined as the sum of the (positively oriented) cells contained in $[0,k]^n$.
        The product of $p$ and $q$ gives $c \in C_{n+1}(\R^n\times\Ucover{X};\R)$ with $\beta(c) = \norm{c}_1 = k^{n+1}$ and $\norm{\Bd{c}}_1 = 2(n+1)k^n$.
        Let $E$ be the set of $n$-cells of $\R^n \times \Ucover{X}$ that are product of an $n$-cell of $\R^n$ with a $0$-cell of $\Ucover{X}$ or of an $(n-1)$-cell of $\R^n$ with a $1$-cell of $\Ucover{X}$.
        Notice that $E$ is partitioned in finitely many $G$-orbits of cells, and that every cell in the support of $\Bd{c}$ belongs to $E$.
        Since $\abs{\beta(c)} = \frac{k}{2(n+1)}\norm{\Bd{c}}_1$ and $k$ can be chosen arbitrarily, any primitive of $\beta$ has to assume arbitrarily large values on the cells of $E$; therefore, it cannot be bounded on orbits.
    \end{proof}
    
    In the proofs that follow, we use the adverb \emph{effectively} with the meaning of ``by means of an algorithm''.
    Also, an effective sequence of objects (\eg, simplicial complexes) $\{X_i\}_{i \in \N}$ is a sequence such that there is an algorithm that outputs $X_i$ when given the number $i \in \N$ as input.
    \begin{theorem}
        Let $n \ge 1$.
        There is no algorithm that, given a finite group-presentation $\langle S|R\rangle$, decides whether $H_\Linfs^n(\langle S|R\rangle;\R) = 0$.
    \end{theorem}
    \begin{proof}
        Suppose by contradiction that there is such an algorithm.
        Given a finite presentation of a group $H$, we can effectively obtain a finite presentation for $G = \Z^{n-1}\times H$ and then check whether $H_\Linfs^n(G;\R) = 0$.
        By applying \Cref{prop:finiteness} to a cellular $K(H,1)$ space $X$ with finite $1$-skeleton, we deduce that $H_\Linfs^n(G;\R) \neq 0$ if $H$ is infinite.
        On the other hand, if $H$ is finite, then $G$ has a finite-index subgroup isomorphic to $\Z^{n-1}$; so $G$ and $\Z^{n-1}$ have isomorphic \Lcoh{} by \Cref{rem:finite_index}.
        But $H_\Linfs^n(\Z^{n-1};\R) = 0$ by dimensional reasons.
        Thus, we are able to effectively check the finiteness of $H$, contradicting the Adian-Rabin theorem applied to the class of finite groups.
    \end{proof}
    
    \begin{theorem}
        \label{thm:undecidable_complexes}
        Let $n,m \in \N$, with either $1 \le n < m$ or $4 \le n = m$.
        There is no algorithm that, given a finite connected $m$-dimensional simplicial complex $X$, decides whether $H_\Linfs^n(X;\R) = 0$. 
    \end{theorem}
    \begin{proof}
        The results of \cite[\S 10]{VKF1974}, due to S.\ P.\ Novikov, can be rephrased as follows: if $d \ge 6$ is an integer, there is an effective sequence $\{Y_i\}_{i\in\N}$ of finite $d$-dimensional simplicial complexes (which actually are manifolds with boundary, but we do not need this) such that:
        \begin{itemize}
            \item Every $Y_i$ is either contractible or has infinite fundamental group;
            \item There is no algorithm that takes a number $i \in \N$ as input and decides whether $Y_i$ is contractible.
        \end{itemize}
        Let us assume by contradiction that there is an algorithm as in the statement.
        Suppose that $1 \le n < m$.
        Fix $d \ge \max\{6, m-n+1\}$, so that there is a sequence of $d$-dimensional complexes as above.
        Consider the following procedure, that can be carried out for every $i \in \N$:
        \begin{enumerate}
            \item Construct $Y_i$;
            \item Compute a triangulation of $(S^1)^{n-1} \times Y_i$ and define $X_i$ to be its $m$-skeleton;
            \item Determine whether $H_\Linfs^n(X_i;\R) = 0$.
        \end{enumerate}
        Notice that $X_i$ has dimension $m$ because $(S^1)^{n-1} \times Y_i$ has dimension $n-1+d \ge m$.
        Since $n < m$, $H_\Linfs^n(X_i;\R)$ is isomorphic to $H_\Linfs^n((S^1)^{n-1} \times Y_i;\R)$, and by \Cref{prop:finiteness} the latter vanishes if and only if the fundamental group of $Y_i$ is finite.
        But this last condition holds if and only if $Y_i$ is contractible.
        Thus, we have a procedure for detecting whether $Y_i$ is contractible; this contradiction settles the case $1 \le n < m$.
        
        Let us assume now that $4 \le n = m$.
        By \Cref{example:manifold_amenable}, if $X$ is a closed $n$-manifold then $H_\Linfs^n(X;\R) = 0$ if and only if the fundamental group of $X$ is not amenable.
        It is well known that, given a finite presentation defining a group $G$, there is an effective way to construct a triangulated closed $n$-dimensional manifold whose fundamental group is isomorphic to $G$.
        The conclusion now follows from the Adian-Rabin theorem applied to the class of amenable groups.
    \end{proof}
    
    The following are the degree-dimension pairs excluded from \Cref{thm:undecidable_complexes}:
    \begin{itemize}
        \item $n = 0$, $m \in \N$.
        In this case, $H_\Linfs^n(X;\R) \cong \R$.
        \item $n = m = 1$.
        In this case, $H_\Linfs^1(X;\R) = 0$ if and only if $X$ is a tree (if it is not a tree, then its fundamental group is infinite and its \Lcoh{} in degree $1$ is nontrivial by \Cref{prop:finiteness}).
        There are easy algorithms to check this condition.
        \item $n = m \in\{2,3\}$. I do not know whether the vanishing of \Lcoh{} is decidable in these cases.
        If $n = m = 2$ it is possible to prove the following characterization: $H_\Linfs^2(X;\R) = 0$ if and only if $X$ is aspherical and its fundamental group is hyperbolic.
        Taken individually, asphericity and hyperbolicity of the fundamental group of a $2$-complex are undecidable (asphericity is considered in \cite{CM1999}; the statement for hyperbolicity follows from the Adian-Rabin theorem), but this is not enough to conclude that the conjunction of the two properties is undecidable.
    \end{itemize}

    \appendix
    \section{Differential forms bounded on orbits}
    \label{sec:differential}
    
    In this appendix we define the \Ldrc{} of a manifold and prove that it is canonically isomorphic to the singular \Lcoh{} defined in \Cref{sec:definition}.
    
    Throughout the appendix, $X$ is a connected smooth manifold of dimension $n$.
    We denote by $G = \pi_1(X, x_0)$ the fundamental group of $X$ with respect to a fixed basepoint, and by $\Ucover{X}$ the universal cover of $X$, on which $G$ acts by covering transformations.
    Let $\pi:\Ucover{X} \to X$ be the covering map.
    To abbreviate the notation, for any subset $U \subseteq X$ we denote by $\Ucover{U} = \pi^{-1}(U)$ its preimage under $\pi$.
    
    Inspired by \Cref{defn:bounded_on_orbits}, we define differential forms bounded on orbits as follows.
    For simplicity, we only consider real-valued differential forms; in general, forms with values in a Banach space could be considered, as in \cite{Mineyev1999}.
    
    \begin{definition}
        Let $U \subseteq X$ be an open subset.
        Fix an auxiliary Riemannian metric on $U$, and lift it to $\Ucover{U}$.
        Let $\omega \in \Omega^k(\Ucover{U})$ be a smooth $k$-form.
        We say that $\omega$ is bounded on orbits if for every $p \in U$ there is an open neighborhood $U_p \subseteq U$ and a number $L_p \ge 0$ such that
        \begin{equation} \label{eq:bounded_on_orbits}
            \abs{\omega(q)(v_1, \ldots, v_k)}\ \le\ L_p \norm{v_1}\cdot \ldots \cdot \norm{v_k}
        \end{equation}
        for every $q \in \Ucover{U_p}$ and tangent vectors $v_1, \ldots, v_k \in T_q(\Ucover{U_p})$.
        For every $k \in \N$ we define
        $\Omega_\Linfs^k(\Ucover{U}) = \{\omega \in \Omega^k(\Ucover{U})\ \mid\ \omega \text{ and } d\omega \text{ are bounded on orbits} \}$.
    \end{definition}
    
    Some remarks:
    
    \begin{itemize}
        \item The choice of the auxiliary Riemannian metric on $U$ is completely irrelevant;
        \item If $\omega \in \Omega^k(\Ucover{U})$ is the pull-back of a $k$-form defined on $U$, then it is bounded on orbits;
        \item If $X$ is compact, then $\omega \in \Omega^k(\Ucover{X})$ is bounded on orbits if and only if it is bounded globally, meaning that there is a constant $L \ge 0$ such that $\abs{\omega(q)(v_1, \ldots, v_k)} \le L \norm{v_1} \cdot \ldots \cdot \norm{v_k}$
        for every $q \in \Ucover{X}$ and $v_1, \ldots, v_k \in T_q(\Ucover{X})$.
    \end{itemize}
    
    \begin{definition}
        Let $X$ be a connected smooth manifold.
        We define the \Ldrc{} of $X$ as the cohomology of the complex $\Omega_\Linfs^\bullet(\Ucover{X})$.
        We denote it by $H_\Ldr^\bullet(X)$.
    \end{definition}
    
    In the rest of the appendix we present a proof of the following theorem.
    Our proof is an adaptation of Bredon's proof of the standard de Rham's theorem described in \cite{Bredon1993}.
    
    \begin{theorem}
        \label{thm:bounded_de_rham}
        Let $X$ be a connected smooth manifold.
        Then there is a canonical isomorphism $H_\Linfs^\bullet(X;\R) \cong H_\Ldr^\bullet(X)$.
    \end{theorem}
    
    Let us fix some notation.
    For any open subset $U \subseteq X$, we denote by $H_\Omega^\bullet(U)$ the cohomology of the complex $\Omega_\Linfs^\bullet(\Ucover{U})$.
    For $U = X$ this coincides with $H_\Ldr^\bullet(X)$.
    We use a different notation because, in general, $\Ucover{U} = \pi^{-1}(U) \subseteq \Ucover{X}$ is not the universal cover of $U$, and therefore $H_\Omega^\bullet(U)$ is not the \Ldrc{} of $U$.
    Notice also that $U$ could be disconnected, and in principle we have defined the \Ldrc{} only for connected manifolds.
    In the same spirit, for every $k \in \N$ we set $C_\Linfs^k(\Ucover{U}) = \{ \alpha \in C^k(\Ucover{U};\R)\ \mid\ \alpha \text{ is bounded on orbits}\footnote{Of course, a singular cochain $\alpha \in C^k(\Ucover{U};\R)$ is bounded on orbits when for every $\sigma \in \Simplex{k}(\Ucover{U})$ the set $\{\alpha(g \cdot \sigma)\ \mid\ g \in G\} \subseteq \R$ is bounded.}\}$ and define $H_C^\bullet(U)$ as the cohomology of the complex $C_\Linfs^\bullet(\Ucover{U})$.
    In particular $H_C^\bullet(X) = H_\Linfs^\bullet(X;\R)$.
    
    \begin{remark}
        \label{rmk:local_coefficients}
        From the point of view of cohomology with local coefficients, $H_\Linfs^\bullet(X;\R)$ is the cohomology of $X$ in a system of local coefficients assigning to any $x \in X$ the $\R[\pi_1(X,x)]$-module $\Linf(\pi_1(X,x),\R)$.
        If $U \subseteq X$ is an open subset, then $H_C^\bullet(U)$ is the cohomology of $U$ in the same system of local coefficients, restricted to $U\subseteq X$.
    \end{remark}
    
    The isomorphism of \Cref{thm:bounded_de_rham} is defined below via integration of differential forms on singular simplices.
    Since the integral can be computed only on smooth simplices, before proceeding we need to introduce a suitable smoothing operator.
    Fix an open set $U \subseteq X$.
    A singular simplex $\sigma:\Simplex{k}\to \Ucover{U}$ is said to be smooth if it admits local smooth extensions on open subsets of the $k$-dimensional hyperplane containing $\Simplex{k} \subseteq \R^{k+1}$.
    As is customary, we denote by $C_\bullet(\Ucover{U};\R)$ the complex of singular chains in $\Ucover{U}$.
    The technical tool we need is a $G$-equivariant chain map $\Sm_U:C_\bullet(\Ucover{U};\R) \to C_\bullet(\Ucover{U};\R)$ such that:
    \begin{itemize}
        \item If $\sigma \in \Simplex{k}(\Ucover{U})$, then $\Sm_U(\sigma)$ is a linear combination of smooth simplices;
        \item If $\sigma\in\Simplex{k}(\Ucover{U})$ is smooth, then $\Sm_U(\sigma) = \sigma$;
        \item $\Sm_U$ is $G$-equivariantly chain homotopic to the identity.
    \end{itemize}
    See, \eg, \cite[Proof of Theorem 18.7]{Lee2012} for the construction of such a chain map.
    The issue of the $G$-equivariance (of the chain map and of the homotopy) is not addressed there, but is easily solved with minor modifications in the proof.
    Now we define a chain map $\mathcal{I}_U:\Omega_\Linfs^\bullet(\Ucover{U}) \to C_\Linfs^\bullet(\Ucover{U})$.
    For any $\omega \in \Omega_\Linfs^k(\Ucover{U})$ we set
    \[ \mathcal{I}_U(\omega)(\sigma)\ =\ \int_\Simplex{k} \sigma^*\omega \]
    if $\sigma \in \Simplex{k}(\Ucover{U})$ is smooth, then extend $\mathcal{I}_U(\omega)$ by linearity on the subspace of $C_k(\Ucover{U};\R)$ generated by smooth simplices, and finally set $\mathcal{I}_U(\omega)(c) = \mathcal{I}_U(\omega)(\Sm_U(c))$ for any singular chain $c \in C_k(\Ucover{U};\R)$.
    The cochain $\mathcal{I}_U(\omega)$ is bounded on the $G$-orbit generated by a smooth $\sigma$: since $\Simplex{k}$ is compact, the constant $L_p$ in (\ref{eq:bounded_on_orbits}) can be choosen independently from $p$ in a neighborhood of $\pi(\sigma(\Simplex{k})) \subseteq U$.
    Boundedness on generic orbits is then implied by the $G$-equivariance of the smoothing operator.
    Finally, the fact that $\mathcal{I}_U$ is a chain map is a consequence of Stokes' theorem.
    
    For every $k \in \N$ we denote by $\mathcal{I}_U^k:H_\Omega^k(U) \to H_C^k(U)$ the induced map in cohomology.
    We will see that these maps are all isomorphisms; in particular, for $U = X$, this will give \Cref{thm:bounded_de_rham}.
    
    \begin{lemma}
        \label{lemma:only_smooth}
        Let $U\subset X$ be an open subset and let $\alpha,\beta \in C_\Linfs^k(\Ucover{U})$ be two singular cocycles assuming the same values on smooth $k$-simplices.
        Then $[\alpha] = [\beta] \in H_C^k(U)$. 
    \end{lemma}
    \begin{proof}
        Let $r:C_\bullet(\Ucover{U};\R) \to C_{\bullet+1}(\Ucover{U};\R)$ be a $G$-equivariant homotopy between $\Sm_U$ and the identity, such that $\Sm_U - \Id = \Bd \circ r + r \circ \Bd$.
        
        The cocycle $\alpha \circ \Sm_U$ is bounded on orbits, because $\alpha$ is bounded on orbits and $\Sm_U$ is $G$-equivariant.
        For the same reason, the cochain $\alpha \circ r$ is bounded on orbits.
        We can write $\alpha - \alpha \circ \Sm_U = \alpha \circ \Bd \circ r + \alpha \circ r \circ \Bd = (\Cbd\alpha)\circ r + \Cbd(\alpha\circ r) = \Cbd(\alpha\circ r)$.
        Therefore, $[\alpha] = [\alpha \circ \Sm_U] \in H_C^k(U)$, and the latter only depends on the values attained by $\alpha$ on smooth simplices.
        The conclusion follows.
    \end{proof}
    
    \Cref{lemma:only_smooth} implies, in particular, that different choices of the smoothing operator $\Sm_U$ lead to the same maps $\mathcal{I}_U^k:H_\Omega^k(U) \to H_C^k(U)$ in cohomology.
    Moreover, these maps commute with the maps induced by inclusions of open subsets.
    In fact, let $U\subseteq V \subseteq X$ be open subsets, and call $i:U\to V$ the inclusion.
    The pull-back via $i$ of a differential form or a singular cochain which is bounded on orbits is bounded on orbits.
    The maps induced by $i$ in cohomology fit in the following diagram:
    \[\begin{tikzcd}
        H_\Omega^k(V) \ar[r,"i^*"] \ar[d,"\mathcal{I}_V^k"] & H_\Omega^k(U) \ar[d,"\mathcal{I}_U^k"] \\
        H_C^k(V) \ar[r,"i^*"] & H_C^k(U).
    \end{tikzcd}\]
    Given any $\omega \in \Omega_\Linfs^k(\Ucover{V})$, it is clear that $(i^* \mathcal{I}_V\omega)(\sigma) = (\mathcal{I}_U i^* \omega)(\sigma)$ for every smooth $\sigma \in \Simplex{k}(\Ucover{U})$.
    Then $[i^*\mathcal{I}_V\omega] = [\mathcal{I}_U i^* \omega] \in H_C^k(U)$ by \Cref{lemma:only_smooth}.
    Hence, the diagram commutes.
    
    In a technical step of the proof of \Cref{thm:bounded_de_rham} we will also need the following refinement of \Cref{lemma:only_smooth}.
    
    \begin{lemma}
        \label{lemma:smooth_and_small}
        Let $U_1, U_2 \subseteq X$ be two open subsets, and let $U = U_1 \cup U_2$.
        We say that a simplex $\sigma \in \Simplex{k}(\Ucover{U})$ is small if its image is contained in $\Ucover{U_1}$ or $\Ucover{U_2}$.
        Let $\alpha,\beta \in C_\Linfs^k(\Ucover{U})$ be two singular cocycles assuming the same values on small smooth $k$-simplices.
        Then $[\alpha] = [\beta] \in H_C^k(U)$.
    \end{lemma}
    \begin{proof}
        The main technical tool is a $G$-equivariant subdivision operator, that is a $G$-equivarant chain map $s:C_\bullet(\Ucover{U};\R) \to C_\bullet(\Ucover{U};\R)$ such that:
        \begin{itemize}
            \item If $\sigma \in \Simplex{k}(\Ucover{U})$, then $s(\sigma)$ is a linear combination of small simplices;
            \item If $\sigma$ is smooth, then $s(\sigma)$ is smooth;
            \item $s - \Id = r \circ \Bd + \Bd \circ r$ for some $G$-equivariant chain homotopy $r:C_\bullet(\Ucover{U};\R) \to C_{\bullet+1}(\Ucover{U};\R)$.
        \end{itemize}
        The existence of such a chain map is based on a barycentric subdivision process, and is the key step in the standard proofs of the excision property or the Mayer-Vietoris theorem for singular homology (see, \eg, \cite[Proposition 2.21]{Hatcher2002}; $G$-equivariance and smoothness are not discussed, but the barycentric subdivision process used there automatically preserves smoothness and gives a $G$-equivariant chain map and homotopy).
        
        Exactly as in the proof of \Cref{lemma:only_smooth}, $\alpha$ and $\alpha \circ s$ represent the same class in $H_C^k(U)$.
        The values attained by $\alpha \circ s$ on smooth simplices only depend on the values attained by $\alpha$ on small smooth simplices.
        An analogous statement holds for $\beta \circ s$.
        By \Cref{lemma:only_smooth}, $[\alpha] = [\alpha \circ s] = [\beta \circ s] = [\beta] \in H_C^k(U)$.
    \end{proof}
    
    The main ingredients for the proof of \Cref{thm:bounded_de_rham} are suitable forms of the Poincaré lemma and of the Mayer-Vietoris sequence for the two cohomology theories $H_C$ and $H_\Omega$.
    Here, we only include the proofs for $H_\Omega$, since the analogous results for $H_C$ follow from the general theory of cohomology with local coefficients (see \Cref{rmk:local_coefficients}).
    We start with the Poincaré lemma.
    
    \begin{proposition}
        \label{prop:poincare}
        Let $U \subseteq X$ be a nonempty open subset, diffeomorphic to a convex open subset $W \subseteq \R^n$.
        Then $H_\Omega^k(U) = 0$ for every $k \ge 1$.
    \end{proposition}
    \begin{proof}
        Without loss of generality we assume that $0 \in W$.
        Since $U$ is simply connected, every component of $\Ucover{U}$ is diffeomorphic to $U$.
        In what follows, $U$ and the components of $\Ucover{U}$ are identified with $W$; in particular they are endowed with the euclidean metric, and the tangent space at every $q \in \Ucover{U}$ is identified with $\R^n$.
        
        Let $\omega \in \Omega_\Linfs^k(\Ucover{U})$ be closed: we have to show that $\omega$ has a primitive which is bounded on orbits.
        We know that for every point in $U$ there is a neighborhood $U'\subseteq U$ and a number $L \ge 0$ such that
        \begin{equation}
            \label{eq:poincare-boo}
            \abs{\omega(q)(v_1, \ldots, v_k)}\ \le\ L \norm{v_1}\cdot\ldots\cdot\norm{v_k}
        \end{equation}
        for every $q \in \Ucover{U'}$ and $v_1, \ldots, v_k \in \R^n$.
        Fix a connected component $U_0 \subseteq \Ucover{U}$ and define a primitive $\mu$ for the restriction of $\omega$ on $U_0$ setting
        \[ \mu(q_0)(v_1, \ldots, v_{k-1})\ =\ \int_0^1 \omega(tq_0)(q_0,tv_1,\ldots,tv_{k-1}) dt \]
        for every $q_0 \in U_0$ and $v_1, \ldots, v_k \in \R^n$.
        This is the same formula used, for instance, in \cite[Chapter V, Lemma 9.2]{Bredon1993} for the proof of the standard Poincaré lemma.
        By compactness, there is an $L \ge 0$ such that the inequality (\ref{eq:poincare-boo}) holds in the neighborhood of the segment joining $0$ and $q_0$.
        Let $M > \norm{q_0}$.
        Then
        \[ \abs{\mu(q_0)(v_1, \ldots, v_{k-1})}\ \le\ \left(\int_0^1 t^{k-1}dt\right) \cdot L M \norm{v_1}\cdot\ldots\cdot\norm{v_{k-1}} \]
        for every $v_1, \ldots, v_{k-1} \in \R^n$.
        The same inequality holds, with the same constants $L$ and $M$, in a neighborhood of $q_0$.
        Moreover, $L$ and $M$ do not depend on the component of $U_0 \subseteq \Ucover{U}$.
        Therefore, defining a primitive of $\omega$ using the same formula in every component, we get a primitive which is bounded on orbits.
    \end{proof}
    
    \begin{proposition}
        \label{prop:mayer_vietoris}
        Let $U_1, U_2 \subseteq X$ be two open subsets, and set $U = U_1 \cup U_2$ and $V = U_1 \cap U_2$.
        Denote by $i_1:\Ucover{U_1} \to \Ucover{U}$, $i_2:\Ucover{U_2} \to \Ucover{U}$, $j_1: \Ucover{V} \to \Ucover{U_1}$ and $j_2:\Ucover{V} \to \Ucover{U_2}$ the inclusions.
        Then for every $k \ge 0$ the sequence
        \[\begin{tikzcd}
            0 \arrow[r] & {\Omega_\Linfs^k(\Ucover{U})} \arrow[r,"{\begin{pmatrix}i_1^*,i_2^*\end{pmatrix}}"] & {\Omega_\Linfs^k(\Ucover{U_1}) \oplus \Omega_\Linfs^k(\Ucover{U_2})} \arrow[r,"j_1^*-j_2^*"] & {\Omega_\Linfs^k(\Ucover{V})} \arrow[r] & 0
        \end{tikzcd}\]
        is exact.
        These sequences induce a long exact sequence in cohomology.
    \end{proposition}
    \begin{proof}
        The composition of the two maps is zero because $i_1\circ j_1$ and $i_2 \circ j_2$ both coincide with the inclusion of $\Ucover{V}$ into $\Ucover{U}$.
        The first map is injective: if the restrictions of $\omega \in \Omega_\Linfs^k(\Ucover{U})$ to $\Ucover{U_1}$ and $\Ucover{U_2}$ are both zero, then $\omega = 0$.
        There is exactness on the middle term: if $\omega_1 \in \Omega_\Linfs^k(\Ucover{U_1})$ and $\omega_2 \in \Omega_\Linfs^k(\Ucover{U_2})$ coincide on the intersection $\Ucover{V}$, then they define a $k$-form $\omega \in \Omega^k(\Ucover{U})$.
        Being bounded on orbits is a local property, hence $\omega \in \Omega_\Linfs^k(\Ucover{U})$.
        
        The surjectivity of the second map is more delicate.
        Fix $\omega \in \Omega_\Linfs^k(\Ucover{V})$.
        Let $\varphi_1,\varphi_2:U \to [0,1]$ be two smooth functions giving a partition of unity, with $\Supp\varphi_1 \subseteq U_1$ and $\Supp\varphi_2 \subseteq U_2$.
        Define $\omega_1 = (\varphi_2 \circ \pi) \cdot \omega$; it is a smooth form in $\Omega^k(\Ucover{U_1})$, vanishing outside $\Ucover{V}$.
        Notice that $\omega_1$ is bounded on orbits, because it is the product of a bounded function with a $k$-form which is bounded on orbits.
        To conclude that $\omega_1 \in \Omega_\Linfs^k(\Ucover{U_1})$ it remains to show that $d\omega_1 = d(\varphi_2\circ \pi)\wedge\omega + (\varphi_2 \circ \pi) \cdot d\omega$ is bounded on orbits.
        The second summand is bounded on orbits, since it is the product of a bounded function with $d\omega$, which is bounded on orbits by assumption.
        As for the first summand, $d(\varphi_2\circ\pi) = \pi^*(d\varphi_2)$ is bounded on orbits because it is a $G$-invariant form, and $\omega$ is bounded on orbits by assumption.
        Hence, for every point in $V$ there is a neighborhood $V'\subseteq V$ and a number $L \ge 0$ such that the inequalities
        \begin{align*}
            \abs{d(\varphi_2 \circ \pi)(q)(v)}\ &\le\ L \norm{v} \\
            \abs{\omega(q)(v_1, \ldots, v_k)}\ &\le\ L \norm{v_1}\cdot\ldots\cdot\norm{v_k}
        \end{align*}
        hold for every $q \in \Ucover{V'}$ and $v, v_1, \ldots, v_k \in T_q(\Ucover{V'})$.
        Now we have
        \begin{align*}
            &\abs{(d(\varphi_2\circ \pi)\wedge \omega)(q)(v_0,\ldots,v_k)}\ =\\
            =\ &\abs{\sum_{i=0}^k (-1)^i\ d(\varphi_2\circ \pi)(q)(v_i)\cdot\omega(q)(v_0,\ldots,v_{i-1},v_{i+1},\ldots,v_k)}\ \le\\
            \le\ &(k+1)\cdot L^2 \cdot \norm{v_0}\cdot\ldots\cdot\norm{v_k}\,,
        \end{align*}
        which means that $d(\varphi_2\circ \pi)\wedge \omega$ is bounded on orbits.
        We conclude that $\omega_1 \in \Omega_\Linfs^k(\Ucover{U_1})$.
        In the same way, $\omega_2 = -(\varphi_1 \circ \pi)\cdot\omega \in \Omega_\Linfs^k(\Ucover{U_2})$, and it is clear that $j_1^*\omega_1 - j_2^*\omega_2 = \omega$.
        
        The maps in the short exact sequence commute with the coboundary operator.
        By standard homological algebra, they induce a long exact sequence in cohomology.
    \end{proof}
    
    We are now ready to prove that $\mathcal{I}_X^k:H_\Omega^k(X) \to H_C^k(X)$ is an isomorphism for every $k \in \N$.
    
    \begin{proof}[Proof of \Cref{thm:bounded_de_rham}]
        Consider the following family of open subsets of $X$: $\mathcal{F}\ =\ \{U \subseteq X \text{ open}\ \mid\ \mathcal{I}^k_U \text{ is an isomorphism for every } k \in \N\}$.
        The goal is to show that $X \in \mathcal{F}$.
        As shown in \cite[Lemma V.9.3]{Bredon1993}, in order to conclude it is enough to prove that $\mathcal{F}$ satisfies the following properties:
        \begin{enumerate}[label=(\arabic*)]
            \item \label{item:F1} If $U \subseteq X$ is an open subset diffeomorphic to an open convex subset of $\R^n$, then $U \in \mathcal{F}$;
            \item \label{item:F2} If $U_1, U_2 \in \mathcal{F}$ and $U_1\cap U_2\in\mathcal{F}$, then $U_1\cup U_2 \in \mathcal{F}$;
            \item \label{item:F3} If $\{U_i\}_i \subseteq \mathcal{F}$ is a family of disjoint open subsets, then $\bigcup_i U_i \in \mathcal{F}$.
        \end{enumerate}
        Property \ref{item:F3} holds because there are canonical isomorphisms
        \begin{align*}
            H_\Omega^k\left(\bigcup_iU_i\right) \cong \prod_iH_\Omega^k(U_i), & & H_C^k\left(\bigcup_iU_i\right) \cong \prod_iH_C^k(U_i).
        \end{align*}
        Let $U \subseteq X$ be an open subset diffeomorphic to an open convex subset of $\R^n$.
        The Poincaré lemma for $H_\Omega$ (\Cref{prop:poincare}) and the analogous result for $H_C$ imply that $\mathcal{I}_U^k$ is (trivially) an isomorphism for every $k \ge 1$.
        On the other hand, $H_\Omega^0(U)$ and $H_C^0(U)$ are both canonically isomorphic to the vector space of locally constant bounded functions $f:\Ucover{U} \to \R$.
        Therefore, property \ref{item:F1} holds.
        
        Let $U_1, U_2 \subseteq X$ be open subsets and denote by $V=U_1\cap U_2$ their intersection and by $U = U_1\cup U_2$ their union.
        The short exact sequences in \Cref{prop:mayer_vietoris} give a long exact sequence in cohomology that, together with the corresponding sequence for $H_C$, fits in the following diagram:
        \[\begin{tikzcd}[column sep=small]
            \cdots \ar[r] & H_\Omega^k(U) \ar[r]\ar[d,"\mathcal{I}_U^k"] & H_\Omega^k(U_1)\oplus H_\Omega^k(U_2)\ar[d,"\mathcal{I}_{U_1}^k\oplus\mathcal{I}_{U_2}^k"] \ar[r] & H_\Omega^k(V) \ar[d,"\mathcal{I}_V^k"] \ar[r] & H_\Omega^{k+1}(U)\ar[d,"\mathcal{I}_U^{k+1}"] \ar[r] & \cdots \\
            \cdots \ar[r] & H_C^k(U) \ar[r] & H_C^k(U_1)\oplus H_C^k(U_2) \ar[r] & H_C^k(V) \ar[r] & H_C^{k+1}(U) \ar[r] & \cdots.
        \end{tikzcd}\]
        Suppose that $U_1, U_2,V \in \mathcal{F}$.
        If we show that the diagram commutes, the five lemma would imply that $U \in \mathcal{F}$, proving property \ref{item:F2} and concluding the proof.
        The leftmost and central square of the diagram above commute because the horizontal maps are induced by inclusions.
        We now consider the square on the right.
        
        As in \Cref{prop:mayer_vietoris}, denote by $i_1, i_2, j_1$ and $j_2$ the various inclusions.
        Let $[\omega] \in H_\Omega^k(V)$.
        Its image in $H_\Omega^{k+1}(U)$ has a representative $\mu$ given by the following procedure: write $\omega = j_1^*\omega_1 - j_2^*\omega_2$ for some $\omega_1 \in \Omega_\Linfs^k(\Ucover{U_1})$ and $\omega_2 \in \Omega_\Linfs^k(\Ucover{U_2})$; then take $\mu \in \Omega_\Linfs^{k+1}(\Ucover{U})$ such that $i_1^*\mu = d\omega_1$ and $i_2^*\mu = d\omega_2$.
        Now consider $\mathcal{I}_V\omega \in C_\Linfs^k(\Ucover{V})$.
        In general $\mathcal{I}_V\omega \neq j_1^*\mathcal{I}_{U_1}\omega_1 - j_2^*\mathcal{I}_{U_2}\omega_2$, but the equality holds on smooth simplices.
        Take a cochain $\alpha \in C_\Linfs^k(\Ucover{U_2})$ such that $\mathcal{I}_V\omega = j_1^*\mathcal{I}_{U_1}\omega_1 - j_2^*(\alpha + \mathcal{I}_{U_2}\omega_2)$, constructed as follows.
        The values attained by $\alpha$ on the simplices contained in $\Ucover{V}$ is uniquely determined (in particular, it must be $0$ for smooth simplices), and we declare that $\alpha$ assigns value $0$ to simplices not contained in $\Ucover{V}$.
        In particular, this choice ensures that $\alpha$ is bounded on orbits.
        The class $[\mathcal{I}_V\omega]$ is sent, via the bottom map, to a class $[\beta] \in H_C^{k+1}(U)$ with $i_1^*\beta=\Cbd\mathcal{I}_{U_1}\omega_1$ and $i_2^*\beta=\Cbd\mathcal{I}_{U_2}\omega_2 + \Cbd\alpha$.
        The cocycles $\beta$ and $\mathcal{I}_U\mu$ assume the same values on smooth simplices contained in either $\Ucover{U_1}$ or $\Ucover{U_2}$.
        Hence, the commutativity of the square follows from \Cref{lemma:smooth_and_small}.
    \end{proof}
    
    \bibliography{bibliography}

\end{document}

%% file: preamble.tex
\usepackage[T1]{fontenc}
\usepackage[final]{microtype}
\usepackage{enumitem}
\usepackage{amsthm, amsmath, amssymb}
\usepackage{tikz-cd}
\usepackage{bm} % bold math
\usepackage{verbatim} % \comment
\usepackage[pdfusetitle,pdfborder={0 0 .5}]{hyperref}
\usepackage{cleveref}
\theoremstyle{plain}
\newtheorem{theorem}{Theorem}[section]
\newtheorem{lemma}[theorem]{Lemma}
\newtheorem{proposition}[theorem]{Proposition}
\newtheorem{corollary}[theorem]{Corollary}

\theoremstyle{definition}
\newtheorem{definition}[theorem]{Definition}
\newtheorem{remark}[theorem]{Remark}
\newtheorem*{remark*}{Remark}
\newtheorem{example}[theorem]{Example}

\newcommand*{\ie}{{\it i.e.}}
\newcommand*{\eg}{{\it e.g.}}

\newcommand*{\Linf}{\ell^\infty}
\newcommand*{\Linfs}{{\scriptscriptstyle{(\infty)}}}
\newcommand*{\Ldr}{{\mathrm{dR}\Linfs}}
\newcommand*{\Lcoh}{$\Linf$\nobreakdash-\hspace{0pt}cohomology}
\newcommand*{\Ldrc}{$\Linf$\nobreakdash-\hspace{0pt}de Rham cohomology}

\newcommand*{\R}{\mathbb{R}} % Real numbers
\newcommand*{\Z}{\mathbb{Z}} % Integers
\newcommand*{\N}{\mathbb{N}} % Natural numbers

\newcommand*{\abs}[1]{\left\lvert#1\right\rvert} % Absolute value
\newcommand*{\norm}[1]{\left\lVert#1\right\rVert} % Norm
\DeclareMathOperator{\Supp}{supp} % Support of a function

\newcommand*{\Id}{\mathrm{id}} % The identity map
\newcommand*{\Ev}{\mathrm{ev}} % Evaluation
\newcommand*{\Ph}{\makebox[1ex]{\textbf{$\cdot$}}} % Placeholder

\newcommand*{\Simplex}[1]{{\Delta^{#1}}}
\newcommand*{\Ucover}[1]{\widetilde{#1}} % Universal cover
\newcommand*{\CW}{CW-complex} % CW complex
\newcommand*{\CWs}{CW-complexes} % CW complexes

\newcommand*{\Bd}[1][{}]{\partial{#1}} % Boundary
\newcommand*{\Cbd}[1][{}]{\delta{#1}} % Coboundary

\newcommand*{\Sm}{\mathrm{sm}} % Smoothing operator